\documentclass[leqno]{amsart}
\usepackage[margin=1in]{geometry}
\usepackage{mathtools}
\usepackage{enumerate}
\usepackage{amsthm}
\usepackage{mathabx}
\usepackage[all]{xy}
\usepackage[titletoc,toc,title]{appendix}
\usepackage[outdir=./]{epstopdf}
\usepackage{graphicx,epstopdf}

\newtheorem{thm}{Theorem}[section] 

\newtheorem{prop}[thm]{Proposition}
\newtheorem{cor}[thm]{Corollary}
\newtheorem{lemma}[thm]{Lemma}
\newtheorem{addendum}[thm]{Addendum}

\theoremstyle{definition}
\newtheorem{Def}{Definition}
\newtheorem*{rems}{Remarks}
\newtheorem{Rem}[thm]{Remark}

\newtheorem*{rem}{Remark}

\theoremstyle{remark}
\newtheorem*{notation}{Notation}
\newtheorem*{example}{Example}

\newtheorem*{ack}{Acknowledgment}

\def\sk5{\vskip 5pt}
\def\newline{\hfil\break}

\def\R{\mathbb R}
\def\C{\mathbb C}

\def\K{\mathbb K}
\def\del{\partial}

\def\sk{\vskip}

\DeclareMathOperator{\upset}{\uparrow}

\numberwithin{equation}{section}

\title{Topological Posets and Tropical Phased Matroids}

\author{Ulysses Alvarez and Ross Geoghegan}

\address{\noindent Ulysses Alvarez, Department of Mathematics, 
University of Alabama, 
Tuscaloosa, AL 35487, USA 
\newline
\vskip 1pt
Ross Geoghegan, Department of Mathematics and Statistics,
Binghamton University (SUNY),
Binghamton, NY 13902-6000, USA}
\email{uaalvarez@ua,edu,
ross@math.binghamton.edu}

\date{May 13, 2024}



\begin{document}
\fontsize{12}{13pt} \selectfont  
    
\begin{abstract}

For a discrete poset $\mathcal X$, McCord proved that the natural
map $|{\mathcal X}|\to {\mathcal X}$, from the order complex to the
poset with the Up topology, is a weak homotopy equivalence. Much later,
\v{Z}ivaljevi\'{c} defined the notion of order complex for a topological
poset. For a large class of topological posets we prove the analog
of McCord's theorem, namely that {\it the natural map from the order
complex to the topological poset with the Up topology is a weak homotopy
equivalence}. A familiar topological  example is the Grassmann poset
$\mathcal{G}_n(\mathbb{\R})$
of proper non-zero linear subspaces of $\R^{n+1}$ partially ordered by
inclusion. But our motivation in topological combinatorics is to apply
the theorem to posets associated with tropical phased matroids over the
tropical phase hyperfield, and in particular to elucidate the tropical version of the
MacPhersonian
Conjecture. This is explained in Section 2.
\end{abstract}

\maketitle

\section{Introduction}\label{intro}
\subsection{A Short Outline}
In this paper we prove a fundamental theorem about topological posets. While our
theorem is
quite general, we are motivated by its application to matroids, specifically to a certain
complex
analog of oriented matroids. These analogous objects are ``tropical phased matroids", a
fairly new notion supported by important foundational work of Baker and Bowler
\cite{BB1}. While the
name sounds exotic (it has connections with tropical geometry through the work of Viro
\cite{Viro1}), the tropical phased matroid is a complex analog of the well-loved oriented
matroid. In
a sense, we are simply moving from something ``real" to something ``complex".

In more detail, let us start with something familiar, a discrete poset $X$. Consider the
associated
order complex of $X$, denoted by $|X|$, i.e. the geometric realization of the abstract
simplicial
complex whose simplices are the finite totally ordered sets in $X$. Long ago, McCord
\cite{Mc1}
proved that the natural  ``comparison map" $|X|\to X$ is a weak homotopy equivalence
when
$X$ is given the ``Up topology" whose open sets are generated by sets of the form
$U_{x}=\{y\mid x\leq y\}$ for all $x\in X$. Much later, McCord's theorem was used by
combinatorialists, for example in \cite{Ba1}, \cite{Ba2}, \cite{An} and \cite{AD}.

After defining topological posets, and limiting ourselves with some extra hypotheses, we
prove
the appropriate analog of McCord's theorem. Of course this requires a definition of the
order
complex (for which we follow \cite{Zi}) and of a comparison map from that
complex to the topological poset with the Up topology.

As explained in \cite{AD}, McCord's theorem is relevant to the MacPhersonian
Conjecture. This
conjecture says, roughly, that the homotopy type of the Grassmannian, 
$\text{Gr}(k, \mathbb{{\R}}^n)$ of linear
$k$-subspaces in ${\R}^n$ can be read off from the (much simpler) discrete poset of all
rank $k$
oriented matroids on the set $\{1,2,\dots ,n\}$. Until recently, it was not possible to have
a
complex analog of this statement because, while the complexification of
$\text{Gr}(k,{\R}^n)$
is clearly
$\text{Gr}(k,{\C}^n)$, it had not been clear what should serve as a suitable replacement
for the
discrete
poset of oriented matroids. We have found that tropical phased matroids fill that gap
nicely. The
place of McCord's Theorem on discrete posets is then filled by our main Comparison
Theorem on topological posets.

To give a clear account of what is hinted at in this introduction, we present in Section $2$
a full
discussion of the relevant aspects of tropical phased matroids.  In particular, we explain 
the complex version of the MacPhersonian Conjecture as well as the broader place of the
conjecture in the relationship between topology and combinatorics. 
It is our hope that the simplification afforded by the Up topology will provide 
new insight into the complex version of the Conjecture, as was already observed  for the
real case
in \cite{AD}. The rest of the paper is then devoted to foundational material on topological
posets needed for our proof of the Comparison Theorem.
\vskip 5pt
We end this short outline by pointing out that this is a paper about topological posets in
general, and we believe that the many technical matters discussed along the way form a
contribution to the literature of such objects, quite independent of its application to
matroids\footnote{For example, Proposition \ref{cover}, Theorem \ref{inverse} and
Theorem \ref{ok}.}.
\vskip 5pt
\begin{ack} We thank Laura Anderson for telling us about this problem, for educating us
on its
connection with tropical phased matroids, and for suggesting the relevance of Hironaka's
Theorem \ref{semi} in showing that some relevant ``tropical sets" are geometric.  
\end{ack}

\vskip 5pt
\subsection{Topological Posets}\label{toppo}
\begin{Def} A {\it topological poset} $({\mathcal X},\leq ,{\mathcal T})$ is a poset
$({\mathcal X},\leq )$ equipped with a Hausdorff topology $\mathcal
T$ on the set ${\mathcal X}$ such that the order relation ${\mathcal P} \coloneqq
\{(x,y)\in
{\mathcal X}\times {\mathcal X}\mid x\leq y\}$ is a closed subspace of
${\mathcal X}\times {\mathcal X}$. 
\end{Def}
\vskip 5pt
A simple example illustrates this: Let $X$ be the subset of the Euclidean plane consisting
of the
unit circle $S^1$ and the origin $0$; the poset inequality is fully defined by $0<x$ for all
$x\in
S^1$. This becomes a topological poset if we remember the usual topology on $X$
inherited
from the plane. The order relation ${\mathcal P}$ is $\{0\}\times X$. We will have more
to say
about this example below in connection with realizations.
\vskip 5pt
\begin{Def} Following \v{Z}ivaljevi\'{c} we say that a topological poset is {\it
mirrored}, and is called an
$M${\it-poset}, if it comes equipped with
a poset map $\mu \colon {\mathcal X}\to {\mathcal R}$, where ${\mathcal R}$
is a finite poset, satisfying
\begin {itemize}
\item $x<y$ implies $\mu (x)<\mu (y)$,
\item for all $r\in {\mathcal R}$, $X_{r} \coloneqq \mu^{-1}(r)$ is a non-empty closed
subset of ${\mathcal X}$.
\end{itemize}
\end{Def}
The map $\mu$ is a {\it mirror} of $\mathcal X$.
\vskip 5pt                              
In our simple example, $\mathcal R$ is $\{1,2\}$; $\mu (0)=1$ and $\mu (S^{1})=2$.
\vskip 5pt
It follows that ${\mathcal X}$ is the topological sum\footnote{The {\it topological sum}
or {\it
coproduct}, 
$\coprod Y_{\alpha}$, of a collection $Y_{\alpha}$ of topological spaces is the disjoint
union
of the spaces $Y_{\alpha}$ with its topology defined by: $U$ is open in $\coprod
Y_{\alpha}$ if
and only if for every $\alpha $, $U\bigcap Y_{\alpha}$ is open in $Y_{\alpha}$. } 
$\coprod _{r\in {\mathcal R}}X_r$.
\vskip 5pt
A less trivial example is the real {\it Grassmann Poset} $\mathcal{G}_{n}({\R})$; here,
${\mathcal X}$
is the set of all {\it proper} linear subspaces of ${\R}^{n+1}$ of positive dimension,
partially
ordered by inclusion, ${\mathcal R}$ is the set $[n]$ of positive integers $\leq n$ with the
natural ordering, and $\mu$ takes the $k$-dimensional subspaces to the integer $k$.

\begin{notation} When $A\subseteq {\mathcal X}$ we define 
$${\upset A} \coloneqq \{x\in {\mathcal X}\mid a\leq x \text{ for some } a\in A\}.$$
\noindent When ${\mathcal A}$ is a family of subsets of $\mathcal X$, we define $\upset
{\mathcal A} \coloneqq \{\upset A\mid A\in \mathcal A\}$, but when $a\in \mathcal X$ 
we write $\upset a$ rather than $\upset \{a\}$.
For $A\subseteq X_r$ and $s>r$, we write $A^{(s)} \coloneqq (\upset A)\bigcap X_s$. 
\end{notation}
\vskip 5pt

Each $X_r$ inherits a topology from ${\mathcal X}$. We choose a basis ${\mathcal
U}_{r}$
for the open sets in $X_r$. 

\begin{Def} We will assume from now on that our $M$-posets have the
{\it Openness Property} that when $U$ is open in $X_r$ then $\upset U$ is open in 
$\mathcal X$; equivalently, for each $s>r$, the
set $U^{(s)}$ is open in $X_s$.  
\end{Def}

\begin{Def} The {\it Up topology} on ${\mathcal X}$ has basis $\bigcup _{r}\{\upset
{\mathcal
U}_{r}\mid
r\in {\mathcal R}\}$. 
\end{Def}
The Openness Property implies that this is a basis for a topology.
\vskip 5pt
It is important to distinguish between the Original topology on ${\mathcal X}$ (which is
Hausdorff) and the Up topology on ${\mathcal X}$ (which is $T_0$ but, in general, not
$T_1$.)
The $M$-posets discussed in this paper are 
considered to have both topologies, the ``Up'' being derived from the ``Original''. The
Openness
Assumption ensures that the Up topology agrees with the Original topology on each
$X_r$; i.e.
the inclusion of $X_r$ into $\mathcal X$ with the Up topology is a topological
embedding.

\begin{example} Let ${\mathcal X}$ be the real Grassmann Poset
$\mathcal{G}_{2}({\R})$. Here, $X_{1}$ is the collection of lines in ${\R}^3$
containing $0$ and $X_{2}$ is the collection of
planes in ${\R}^3$ containing $0$, partially ordered by inclusion. It is instructive to
consider the
two topologies in this situation: 
\begin{itemize}
\item {\it The Ordinary topology:} We indicate a basis. Given a line $\ell$ in ${\R}^3$
and a
number $\epsilon >0$, a basic neighborhood of $\ell$ in $X_1$ consists of all lines in
${\R}^3$
whose angle with $\ell$ is less than $\epsilon$. Given a plain $p$ in ${\R}^3$ and a
number
$\epsilon >0$, a basic neighborhood of $p$ in $X_2$ consists of all plains in ${\R}^3$
whose
angle with $p$ is less than $\epsilon$. Then $\mathcal X$ with the Ordinary topology is
just
$X_{1}\coprod X_{2}$.  
\item {\it The Up topology:} Again, we indicate a basis. Given a line $\ell$ in ${\R}^3$,
a basic 
``Up-neighborhood" $W$ of $\ell$ in $\mathcal X$ consists of a basic neighborhood
$W_1$ of
$\ell$ in
$X_1$ together with all planes in ${\R}^3$ that contain a line lying in $W_1$. Given a
plane
$p$
in $X_2$, a basic ``Up-neighborhood" of $p$ is just a basic neighborhood of $p$ in the
Ordinary
topology. 
\end{itemize}
\end{example}
\vskip 5pt
\subsection {The Order Complex and the Comparison Map.}\label{order} 
In \cite{Zi}, a
simplicial space $\Delta ({\mathcal X})$, called the order complex\footnote{Strictly
speaking, this is the geometric realization of a simplicial
topological
space.}, is associated with each topological poset ${\mathcal X}$.
For $M$-posets where the spaces $X_r$ are locally compact polyhedra (as will always be
the
case in this paper) there is a rather simple definition as follows:
\vskip 5pt
\begin{Def}Each element $z$ of the topological join ${\Asterisk_{r\in {\mathcal
R}}}X_r$ has the form 
$z={\sum}t _{i}x_{i}$ where $x_i\in X_i, 
t _{i}\geq 0, \text{ and } \sum t_i=1$. Define
the {\it support} of $z$ to be 
$\text{supp}(z) \coloneqq \{i\in {\mathcal R}\mid t _{i}>0\}$. Then the {\it order
complex}
is 
$$\Delta ({\mathcal X}) \coloneqq 
\{z\in {\Asterisk_{r\in {\mathcal R}}}X_r\mid\text{supp}(z)\text{ is
a chain in }{\mathcal R}, \text{ and if }i<j \text{ in supp}(z) \text{ then } 
x_{i}<x_{j} \text{ in
}{\mathcal X}\}.$$
\end{Def}

\noindent We will sometimes identify $X_r$ with the subspace of $\Delta ({\mathcal
X})$
whose support is the singleton $\{r\}$.
\begin{example}
We return to our simple example, above, where $\mathcal X$ is $S^{1}\coprod \{0\}$. If
$X$ is
considered as a topological poset, then its realization $\Delta ({\mathcal X})$ is
homeomorphic
to a planar disk. But if $\mathcal X$ is just a discrete poset then its realization
$|{\mathcal X|}$
is a $1$-dimensional complex consisting of uncountably many copies of the unit interval
with
their $0$-points identified. 


\end{example}  
\begin{example} Vassiliev \cite{V1}, \cite{V2} proved that for $\mathbb K$ the reals,
the
complex numbers, or the quaternions, and 
$\mathcal{G}_{n}({\mathbb K})$ the Grassmann Poset of proper non-zero linear
subspaces of ${\mathbb K}^{n+1}$, the corresponding order complex, 
$\Delta (\mathcal{G}_{n}({\mathbb K}))$, is homeomorphic to the sphere $S^m$ where
$m={n+1 \choose
2}d+n-1$, $d$ being the dimension of $\mathbb K$ over $\mathbb R$.
\end{example}

\begin{Def} The {\it Comparison Map} $f \colon \Delta ({\mathcal X})\to {\mathcal X}$
is
defined by
$f(z)=\text{max}\{x_{i}\mid i\in \text{ supp}(z)\}$. It is not hard to prove that $f$ is
continuous\footnote{The word ``map'' is used in this paper to mean that the given
function is
continuous.} with respect to the Up
topology on ${\mathcal X}$.
\end{Def}
\vskip 5pt
Our aim is to prove that under reasonable topological and
geometric assumptions, and considering $\mathcal X$ with the Up topology,
the Comparison Map $f$ is a weak homotopy equivalence.
\vskip 5pt
\subsection{Polyhedral $M$-posets}
\begin{Def} An $M$-poset ${\mathcal X}$ is {\it polyhedral} if each $X_i$ is a locally
compact
polyhedron\footnote{In Section \ref{alpha} we go into detail about our use
of polyhedral language.}, and $\mathcal X$ has the {\it Polyhedral Diagonal Property},
namely:
The order relation $\mathcal P$ is a (closed)  subpolyhedron of ${\mathcal X}\times
{\mathcal
X}$; here $\mathcal X$ has the Original topology.
\end{Def} 

\vskip 5pt
\subsection{Convenient Choice of Bases}\label{convenient} When the space $X_r$ is
polyhedral each point has a basic
system of compact neighborhoods which are piecewise linear (PL) cones on their
frontiers\footnote{See Section \ref{alpha}.}. We may assume that the basis ${\mathcal
U}_r$ of $X_r$ 
consists of the interiors of such cones. In particular, each such set is contractible and its
closure is
a compact contractible subpolyhedron of $X_r$. {\it Our basis ${\mathcal
U}=\bigcup _{r}\{\upset {\mathcal U}_{r}\}$ for the topology of $\mathcal X$
will always be understood to consist of such sets.}

\vskip 5pt
\subsection{Geometric Posets and Main Theorem}\label{geom}
\begin{Def} A compact subset $C$ of a polyhedron $D$ has {\it trivial shape} if $C$
can be contracted to a point in any of its neighborhoods\footnote{Shape Theory is
discussed in
Appendix \ref{appendix}.} in $D$.
\end{Def} 
\noindent This
is known to be an intrinsic property of $C$, independent of $D$. Note that when $C$
is contractible then $C$ has trivial shape, but trivial shape is more
general: for example, the Topologist's Sine Curve in ${\R}^2$ has trivial
shape but has two path components.
\vskip 5pt
\begin{notation} When $(M,d)$ is a metric space, the space of non-empty compact
subsets of $M$ with the Hausdorff metric is denoted by $cM$.
\end{notation}
\vskip 5pt
\begin{Def}\label{geometric} A polyhedral $M$-poset $({\mathcal X},\leq ,{\mathcal
T})$ is
{\it geometric} if, for each $r<s\in {\mathcal R}$ it has the
following 
\newline {\bf Additional Properties:}
\vskip 5pt
\begin{enumerate}[(1)]
\item [A1]: When $x\in X_r$, the set $x^{(s)}$ is compact and non-empty.
\item [A2]: The map $\pi _{r,s} \colon X_{r}\to cX_{s}$ defined by 
$x\mapsto x^{(s)}$ is continuous.
\item [A3]: When  $U\in {\mathcal U}_r$ and $y \in \overline
U^{(s)}$, then $\{x\in {\overline U}\mid x<y\}$ has trivial
shape\footnote{This implies that for every such $y$ there is some $x<y$. This is because
the
empty space does not have trivial shape.}.
\end{enumerate}
\end{Def}
\vskip 5pt
We note that ``geometric" only involves the Original topology.
\begin{thm}\label{main} {\bf (Comparison Theorem)} When the $M$-poset ${\mathcal
X}$ is
geometric and carries the Up topology,  the map
$f \colon \Delta ({\mathcal X})\to {\mathcal X}$ is a weak homotopy equivalence.  
\end{thm} 
\vskip 5pt
The discrete case of Theorem \ref{main} is due to McCord \cite{Mc2}. See Section
\ref{discrete} for a discussion.
\vskip 5pt
\begin{example} The Grassmann Posets $\mathcal{G}_{n}({\mathbb K})$ are
geometric, so
Theorem \ref{main} and Vassiliev's results give us the singular homology of ${\mathcal
G}_{n}(\mathbb K)$ with the Up topology.
\end{example}
\subsection{Known theorem on weak homotopy equivalences}
Our proof of Theorem \ref{main} uses the following theorem of McCord \cite{Mc1},
reproved independently\footnote{May's paper appears to date from the late 1970's though
the publication year is 2006.} by May \cite{M}.

\begin{thm}\label{mccord} Let $g \colon C\to D$ be a map between topological spaces,
and let
${\mathcal B}$ be a basis for the open subsets of $D$. If for all $B\in {\mathcal B}$ the
restriction $g| \colon g^{-1}(B)\to B$ is a weak homotopy equivalence, then $g$ is a
weak homotopy equivalence.  
\end{thm}

We will be prove that for $U\in {\mathcal U}_r$ both $\upset
U$ and $f^{-1}(\upset U)$ are weakly contractible\footnote{i.e. all
homotopy groups are trivial.}, so that we can apply Theorem \ref{mccord}.

\vskip 5pt
\section{Matroid Examples}
\subsection{Introduction}
In this section we describe some examples of geometric $M$-posets and applications of
Theorem \ref{main}. These lead up to a discussion of what is often called the
MacPhersonian Conjecture. To present these examples we draw on the work of Baker and
Bowler \cite{BB1}, but we try to include enough basic information that the section is
more
or less self-contained.
\vskip 5pt

\subsection{The tropical phase hyperfield}
Let $T\Phi$ denote the poset $\{0\} \cup S^{1} \subseteq \mathbb{C}$, where the partial
ordering is fully defined by $0<x$ for all $x\in S^1$. This is an $M$-poset over
$\{1,2\}$, where
the mirror $\mu$ takes $0$ to $1$ and all of $S^1$ to $2$. The Openness Property holds
and
this is easily seen to be a geometric poset.
\vskip 5pt 
$T\Phi $ also carries the structure of a hyperfield. A {\it hyperfield} is similar in
definition to a
field (associative, commutative etc.) except that the addition operation is allowed to be
multivalued. In the case of $T\Phi$, multiplication is inherited from the complex plane
and the
sum $\boxplus$ is defined by
\begin{itemize}
\item $x \boxplus 0=\{x\}$,
\item $x \boxplus -x=\{0\} \cup S^1$ whenever $x \neq 0$,
\item $x \boxplus y$ is the smallest closed arc in $S^1$ joining the points $x$ and $y$
when $y \neq -x$,
\item For subsets $A$ and $B$, $A\boxplus B=\bigcup _{a\in A, b\in B}a\boxplus b$.
\end{itemize}
$T\Phi $ is the {\it tropical phase hyperfield} introduced by Viro in
\cite{Viro1}\footnote{Not to
be confused with the {\it phase hyperfield} where $x\boxplus -x=\{0,x,-x\}$
and ``closed arc" is replaced by ``open arc".}. It is a topological hyperfield in the sense of
\cite{AD}.
\vskip 5pt
\begin{rem} The hyperfield $T\Phi$ is the complex analog of the ``sign
hyperfield'' ${\mathbb S}\coloneqq T\Phi \bigcap {\mathbb R}$. Just as $T\Phi$
will lead us to a discussion of tropical phased matroids, $\mathbb S$
would lead to an analogous discussion of oriented matroids. However,
unlike $T\Phi$, the original topology on $\mathbb S$ is discrete, so the
analog of the topological posets discussed here would be discrete posets, and everything 
would be much simpler.
\end{rem}
\subsection{Various relevant $M$-posets}

\subsubsection{Products} The product space $T\Phi ^{n}$ (using the
Original topology) is a topological poset, where the comparison $\leq$
is defined coordinate-wise, i.e. $(x_{1},\dots , x_{n})\leq (y_{1},\dots
,y_{n})$ if and only if $x_{i}\leq y_{i}$ for all $i$. This has a
 minimal element $\mathbf 0$. More interesting are products
with zero deleted. Consider $T\Phi
 ^{n}-\{\mathbf 0\}$. The mirror
map $\mu :T\Phi ^{n}-\{\mathbf 0\}\to [n]$ is defined
 by\footnote{In
the context of product posets, the {\it support}
of an $n$-tuple is the set of its non-zero entries.}
$\mu(\mathbf{v})=|\text{support}(\mathbf{v})|$. This $M$-poset is polyhedral because the pre-image
of each $i$ is homeomorphic to a disjoint union of tori.
 That it has the Polyhedral Diagonal Property
is then clear in view of the following remark.

\vskip 5pt
\begin{Rem}\label{polyprop} One usually checks the Polyhedral Diagonal Property as
follows: There is a known polyhedral $M$-poset $\mathcal X$ with (polyhedral) order
relation ${\mathcal P}$, and one is interested in a subset $\mathcal Y$ such that  each
$Y_i$ is a polyhedral subset of $X_i$. Then the order relation for $\mathcal Y$ is the
polyhedron 
${\mathcal P}\bigcap ({\mathcal Y}\times {\mathcal Y})$, so $\mathcal Y$ is a
polyhedral
$M$-poset.
\end{Rem}
\vskip 5pt
\begin{prop}\label{product}
The polyhedral $M$-poset $T\Phi ^{n}-\{\mathbf 0\}$ is geometric.
\end{prop}
\begin{proof}
The Additional Properties A1 and A2 are clear. Recall that Property A3
refers to our convenient choice of basis.  Write ${\mathcal X}\coloneqq\Phi
^{n}-\{\mathbf 0\}$. Let $r<s$, let $U$ be a basic open set in $X_r$,
and let ${\mathbf y}\in {\overline U}^{(s)}$. There is a sequence $\mathbf y_{m}$
in $ U^{(s)}$ converging to $\mathbf y$, so (renaming points as necessary) there is a
corresponding
sequence ${\mathbf x_{m}}$ in $U$ converging to some
${\mathbf x}\in \overline U$
where, for each $m$, ${\mathbf x_{m}}<{\mathbf y_{m}}$. Since
the order relation $\mathcal P$ is closed, and contains each pair 
$({\mathbf x_{m}}, {\mathbf y_{m}})$ it must also contain 
$({\mathbf x}, {\mathbf y})$. Thus 
$\{{\mathbf x}\in {\overline U}\mid {\mathbf x}<{\mathbf y}\}$ is non-empty. There
are only
finitely many points ${\mathbf z}\in X_r$ satisfying ${\mathbf z}<{\mathbf y}$,
each being in a different component of $X_r$, so $\mathbf x$ is the only one  
in the component of $X_r$ containing $\overline U$. So $\{{\mathbf z}\in {\overline
U}\mid
{\mathbf z}<{\mathbf y}\}$ is a single point.
\end{proof} 
\vskip 5pt
To prove polyhedrality in our other $M$-posets we will use the following well-known
theorem:
\begin{thm}\label{semi}(\cite{Hir}) Let $S\subseteq {\mathbb R}^n$ be a semialgebraic
set (i.e.
the solution set of a finite number of polynomial equalities and polynomial inequalities, or
a
finite union of such). Then $S$ admits a (canonical) PL triangulation.
\end{thm}
\subsubsection{Covectors}
For an $n$-vector $\mathbf{v}=(v_1,\ldots,v_n) \in T\Phi^{n}-\{\mathbf 0\}$, its {\it
covector} is   
$$\mathbf{v}^\perp\coloneq\{(x_1,\ldots,x_n) \in T\Phi^{n}-\{\mathbf 0\}\mid 0 \in
v_{1}x_1
\boxplus \ldots \boxplus v_{n}x_n\}.$$
\noindent
The set ${\mathbf v}^{\perp}$ is a subset of the $M$-poset $T\Phi ^{n}-\{\mathbf 0\}$,
so it is also an $M$-poset, provided the mirror map is taken to be $\mu :{\mathbf v}^{\perp}\to
\{2, 3, \dots, n\}$ (because, with this restriction, $\mu ^{-1}(1)$ is empty). 
\vskip 5pt
\begin{thm}\label{perp} The $M$-poset ${\mathbf v}^{\perp}$ is geometric.
\end{thm}
\begin{proof}
We first prove polyhedrality. Because $(S^1)^n$ acts transitively
it is enough to consider the case $\bf{v=1}$, i.e. to show that the set
$S:=\{(x_{1},\ldots ,x_{n})\mid 0\in x_1\boxplus \cdots \boxplus x_n\}$
is polyhedral. By induction we may assume the points $x_{1},\ldots , x_{n}$
lie in $S^1$ and are distinct, since otherwise, the case would have been dealt with for a 
value lower than $n$. As we have said, the case $n=1$ does not occur. When $n=2$ the
set $S$ is a copy of $S^1$ and hence is polyhedral. When $n=3$ a point $(x_{1}, x_{2},
x_{3})$
lies in $S$ if and only if the closed convex hull in $\C$ having these
points as vertices contains $0$. To use Theorem \ref{semi}, we regard the points $x_i$ as
vectors in the plane $\C$. For each $i$,
let $y_{i}$ be a unit vector orthogonal to $x_i$. The condition then
is that any two of the following hold (the third being a consequence of
	the other two)\footnote{Geometrically this says that if we denote by $d_i$ the diameter 
$[x_{i},-x_{i}]$ in the unit disk, then $x_j$ and $x_k$ must be on opposite sides of
$d_i$,
and this must hold for all three choices of subscripts, though if two hold then so does the
third.)}:
\begin{enumerate}
\item $(y_{1} \cdot x_{2})( y_{1} \cdot x_{3})\leq 0$;
\item $(y_{2} \cdot x_{1})( y_{2} \cdot x_{3})\leq 0$;
\item $(y_{3} \cdot x_{1})( y_{3} \cdot x_{2})\leq 0$.
\end{enumerate}
When $n>3$ the condition must hold for some triple of points, so polyhedrality is proved
for all
$n$. The Additional Properties A1 and A2 clearly hold, and A3 holds for the same
reasons as in
the proof of Proposition \ref{product}.  
\end{proof}
\vskip 5pt
\begin{rem}In \cite{Al21}, the first-named author proves that the order
complex $\Delta({\mathbf v}^\perp)$ is homeomorphic to the sphere $S^{2n-3}$.
Together with
Theorems \ref{main} and \ref{perp}, this implies that the singular
homology of ${\bf {v}}^\perp $ with the Up topology is that of
$S^{2n-3}$.
\end{rem}
\vskip 5pt
\subsection{Strong matroids over hyperfields}
Other than fields, the only hyperfield occurring in this paper is $T\Phi$. The general
definition of a hyperfield can be found, for example, in \cite{BB1} or \cite{BB2}. To
simplify notation in this subsection it is convenient to state things
in terms of an arbitrary hyperfield $F$, keeping in mind that, for us,
$F$ will  always be $\C$ or $T\Phi$.
\vskip 5pt
Given positive integers $n$ and $r$ and a hyperfield $F$, a {\it strong
Grassmann-Pl{\"u}cker function of rank} $r$ on $[n]$ with coefficients in $F$ is a
function
$\varphi :[n]^{r}\to F$ such that 
\begin{enumerate}[(1)]
\item $\varphi $ is not identically zero;
\item $\varphi $ is alternating, i.e. $\varphi (x_{1},\dots ,{x}_{i},\dots ,x_{j},\dots
 ,x_{r})=-\varphi (x_{1},\dots ,{x}_{j},\dots, x_{i},\dots ,x_{r})$; 
\item $\varphi $ satisfies the strong Grassman-Pl{\"u}cker relations: When
$\{x_1,\dots,x_{r+1}\}\subseteq [n]$ and $\{y_1,\dots ,y_{r-1}\}\subseteq [n]$ then 
$$0\in \boxplus _{k=1}^{r+1}(-1)^{k}\varphi (x_{1},\dots ,{\hat x}_{k},\dots
,x_{r+1}). 
\varphi (x_{k},y_{1},\dots ,y_{r-1}).$$
\end{enumerate}
\vskip 5pt
Two such functions $\varphi _1$ and $\varphi _2$ are {\it equivalent} if
$\varphi _1=t.\varphi _2$ where $t\in F-\{0\}$, i.e. one is a non-zero scalar multiple
of the other.
\vskip 5pt
Strong $F$-matroids are defined in \cite{BB1}. For our purposes we will instead use the
following Theorem 3.13 of \cite{BB1} as our {\it definition} of ``strong $F$-matroid".
\vskip 5pt 
\begin{thm}\label{strong} There is a natural bijection between the set of equivalence
classes of
strong  Grassman-Pl{\"u}cker functions of rank $r$ on $[n]$ with coefficients in $F$ and
strong 
$F$-matroids of rank $r$ on $[n]$.
\end{thm}
A parallel discussion is possible for ``weak Grassman-Pl{\"u}cker functions"; see
\cite{BB1}
and \cite{BB2} for the definitions. 
\subsection {The space of strong Grassmann-Pl{\"u}cker functions}
A map $\varphi :[n]^{r}\to T\Phi $ assigns an element 
$\varphi (x_{1},\ldots ,x_{r})\in T\Phi$ to
each $r$-tuple $(x_{1},\ldots ,x_{r})$ of elements of $[n]$. When we order the
$r$-tuples
lexicographically, $\varphi $ is encoded by an element of $T\Phi ^m$ where 
$m=\frac{n!}{(n-r)!}$. We write $\text{GP}_{s}(r, T\Phi ^{n})$ for the space of all
strong
Grassmann-Pl{\"u}cker functions of rank $r$ on the set $[n]$ with coefficients in
$T\Phi$. This
is the subset of $T\Phi ^{m}$ consisting of elements satisfying the non-zero and
alternating
conditions as well as (3) in the definition above.
\begin{prop}\label{GP}
The $M$-poset $\text{\rm GP}_{s}(r, T\Phi ^{n})$ is semialgebraic.
\end{prop}
\begin{proof}
We prove this using Theorem \ref{semi}. The first two conditions are clearly
semialgebraic. For fixed $\{x_1,\dots,x_{r+1}\}\subseteq [n]$ and
$\{y_1,\dots,y_{r-1}\}\subseteq [n]$ we consider the condition  
$$0\in \boxplus _{k=1}^{r+1}(-1)^{k}\varphi (x_{1},\dots ,{\hat x}_{k},\dots
,x_{r+1}). 
\varphi (x_{k},y_{1},\dots ,y_{r-1})$$
for all $\varphi $ mapping $[n]^{r}$ into $S^1$. The set of such $\varphi $ is
semialgebraic for
the reasons\footnote{Note that the passage from $u_{k}:=\varphi (x_{1},\dots ,{\hat
x}_{k},\dots , x_{r+1})$ and $v_{k}:=\varphi (x_{k}, y_{1},\dots , y_{r-1})$ to 
$u_{k}v_{k}$ is semialgebraic.} given in the proof of Theorem \ref{perp}. As 
$\{x_1,\dots,x_{r+1}\}\subseteq [n]$ and $\{y_1,\dots,y_{r-1}\}\subseteq [n]$ vary this
gives a
finite union of semialgebraic sets. This discussion can easily be adapted to the cases
where some
of the $\varphi $-values are $0\in T\Phi $.
\end{proof} 
\subsection{Strong tropical phased Grassmannians}\label{circuit}
The {\it strong Grassmannian}, $\text{Gr}_{s}(r,T\Phi ^{n})$, is defined in the
literature\footnote{See \cite{AD} for more details.} to be the set of all strong
tropical phased matroids of rank $r$ on the set $[n]$ with
coefficients in $T\Phi $. In other words, 
$\text{Gr}_{s}(r,T\Phi ^{n})$ is the quotient space $\text{GP}_{s}(r, T\Phi
^{n})/S^{1}$.
\begin{thm}\label{Grassmannian} The $M$-poset $\text{Gr}_{s}(r,T\Phi ^{n})$ is
geometric.
\end{thm}
\begin{proof}
We first prove that $\text{\rm Gr}_{s}(r,T\Phi ^{n})$ is semialgebraic
hence polyhedral. 
\vskip 5pt
Once an ordering is chosen for the set $[n]^{r}$, a point of
$\text{GP}_{s}(r, T\Phi ^{n})$ is an ordered set of $r$-tuples. In this
case it is more convenient to forget the $r$-tuple partition, and to
consider that point as just an element of $T\Phi ^{m}\subseteq {\C}^m$
where $m=\frac{n!}{(n-r)!}$. By Proposition \ref{GP} $\text{GP}_{s}(r, T\Phi
^{n})$ can be considered as a semialgebraic subset of ${\C}^m$.
We are to prove that the quotient of this by
$S^1$-multiplication does not destroy the semialgebraic property.
\vskip 5pt
Let $C_1$ denote the subset consisting
of all points whose last entry lies in $S^1$. This is a union of
components of a semialgebraic set, and hence is semialgebraic. For each
$p\in C_1$ there is $t\in S^1$ such that $t.p$ has last entry $1\in
S^1$. We denote by $C_{1,1}$ the subset of $C_1$ consisting of  points
whose last entry is $1$. This is easily seen to be a semialgebraic
set, hence polyhedral. It follows that $C_1$, being homeomorphic to
$S^{1}\times C_{1,1}$, is polyhedral. We consider $S^{1}\times C_{1,1}$
to be an $S^1$-space where the $S^1$-action is by rotation on the first
factor and is trivial on the second factor; and, of course, $C_1$ is an
$S^1$-space under multiplication by scalars. Writing an arbitrary element
of $C_1$ as $({\mathbf b},x)$, where $\mathbf b$ involves all the entries
except the last, and $x\in S^1$ is the last entry, an $S^1$-equivariant
homeomorphism $C_{1}\to S^{1}\times C_{1,1}$ is given by $({\mathbf
b},x)\mapsto (x,(x^{-1}{\mathbf b},1))$.
\vskip 5pt
Next, let $C_2$ denote the subset consisting of all points whose last
entry is $0$ and whose second-to-last entry lies in $S^1$. By the
same argument, $C_2$ is polyhedral, and is equivariantly homeomorphic
to $S^{1}\times C_{2,1}$. The pattern is now clear. The polyhedral
sets $C_i$ are clopen and pairwise disjoint. Since all the points
of $\text{GP}_{s}(r, T\Phi ^{n})$ have some non-zero entry we see
that $\bigcup _{i}C_{i}=\text{GP}_{s}(r, T\Phi ^{n})$ which 
is equivariantly homeomorphic to $\bigcup _{i}S^{1}\times
C_{i,1}$. Factoring out the $S^1$-actions gives a homeomorphism between
$\text{Gr}_{s}(r,T\Phi ^{n})$ and the semialgebraic set $\bigcup
_{i}C_{i,1}$. This proves polyhedrality.
\vskip 5pt
The Additional Property A3 holds for the same reason as in the proof of Proposition
\ref{product}. The ``non-empty" requirement in A1 holds because the condition $0 \in
x_{1}
\boxplus \cdots \boxplus x_{n}$ continues to hold if a $0$-entry $x_i$ is replaced by a
non-zero
entry. The proof of Property A2 is left to the reader.
\end{proof}
\vskip 5pt
The ordinary complex Grassmannian $\text{Gr}(r,{\mathbb C}^{n})$ is the
space of $r$-dimensional linear subspaces of ${\mathbb C}^{n}$; here $1\leq r\leq n-1$.
But it
can also be viewed in the above terms, since every field is a hyperfield (in which case
$a+b=\{a+b\}$). Thus $\text{Gr}(r,{\mathbb C}^{n})$ can be identified with the
quotient space,
modulo scalar multiplication, of the space of Grassmann-Pl{\"u}cker
functions\footnote{Over a
field, strong and weak are the same.} of rank $r$ on $[n]$ with coefficients in $\mathbb
C$. The
map ${\mathbb C}\to T\Phi$ sending non-zero $z$ to
$\frac{z}{|z|}$ and $0$ to $0$ is continuous with respect to the Up
topology. It induces the map $\rho _{s}:\text{Gr}(r, {\mathbb C}^{n})\to
\text{Gr}_{s}(r, T\Phi^{n})$, which is also continuous when the $M$-poset 
$\text{Gr}_{s}(r, T\Phi^{n})$ has the Up topology. 
\vskip 5pt

\subsection{The Tropical MacPhersonian Conjecture}
The following is the tropical phased analog of a well-known conjecture for oriented
matroids:
\vskip 5pt
{\bf Question:} Is the map $\rho _{s}:\text{\rm Gr}(r, {\mathbb C}^{n})\to \text{\rm
Gr}_{s}(r,
T\Phi^{n})$ a weak homotopy equivalence?
\vskip 5pt
\begin{rems}
\begin{enumerate}[(i)]
\item There is a homotopy commutative diagram 
$$
\xymatrix{
&\Delta (\text{Gr}_{s}(r,T\Phi^{n}))\ar[d]^{f}\\
\text{Gr}(r,{\mathbb C}^{n})\ar[r]_{\rho _{s}}\ar[ur]^{\Delta (\rho
_{s})}&\text{Gr}_{s}(r,T\Phi^{n})
}
$$
In view of Theorem \ref{main}, the map $\rho _{s}$ is a weak homotopy
equivalence if and only if $\Delta (\rho _{s})$ is. The traditional form
of the Question is for $\Delta (\rho _{s})$. We believe the problem will be
more tractable using the poset itself with the Up topology.
\item The analogous ``MacPhersonian Conjecture" for oriented matroids
(still open having been incorrectly dealt with in \cite{biss}) says
that the analogous ``real" map $\Delta (\rho):\text{Gr}(r, {\mathbb
R}^{n})\to \Delta (\text{Gr}(r, {\mathbb S}^{n}))$ is a weak homtopy
equivalence. Since $\text{Gr}(r,{\mathbb S}^{n}))$ is a finite
discrete $M$-poset, the relevant analog of Theorem \ref{main} was proved long
ago by McCord (see Section \ref{discrete}), and it was already pointed
out in \cite{AD} that the MacPhersonian Conjecture is the same whether worded for
$\rho $
or for $\Delta (\rho)$.
\end{enumerate}
\end{rems}

\subsection{The place of the MacPhersonian Conjecture} The Conjecture lies at the
border between topology and combinatorics. Let ${\K}$ denote either of the fields
${\R}$ or ${\C}$. The classifying space for  $r$-dimensional ${\K}$-vector bundles is 
$\text{\rm Gr}(r, {\K}^{\infty})$, the direct limit over $n$ of spaces 
$\text{\rm Gr}(r, {\K}^{n})$. This means that there is a natural bijection between the set
of isomorphism classes of  $r$-dimensional ${\K}$-vector bundles over a space $B$ and
the set of homotopy classes of maps  
from $B$ to $\text{\rm Gr}(r, {\K}^{\infty})$. The point of the MacPhersonian
Conjecture, if true, is that $\text{\rm Gr}(r, {\K}^{n})$ is faithfully modeled  by a
combinatorial object, namely $\text{\rm Gr}(r, F^{n})$, where the hyperfield $F$ is
$T\Phi $ in the complex case, or the sign hyperfield ${\mathbb S}$ in the real case,
leading,
in the limit, to a combinatorial model for the classifying space.
\vskip 5pt
The rest of this paper is concerned with the proof of Theorem \ref{main}.
\vskip 5pt

\section{Preliminaries}

\subsection{Polyhedra.}\label{alpha} This is for reference. We need to be precise about
our
polyhedral terminology.
\vskip 5pt
A {\it rectilinear simplicial complex} consists of a locally finite, finite-dimensional,
abstract
simplicial complex $K$, whose geometric realization is denoted by $|K|$. The space
$|K|$ is
embedded as a subset of a Euclidean space $E$ in such a way that each simplex is a
convex
subset of $E$. Moreover, the embedding is a closed proper\footnote{A map is {\it
proper} if the
pre-image of every compact subset of its codomain is compact.} map $|K|\to E$. The
choice of
$E$ and of the embedding will be suppressed, and from now on we will denote the image
of
$|K|$ simply by $K$, omitting the vertical bars. 
\vskip 5pt
The Euclidean space $E$ should carry either the $\ell _1$ metric or the $\ell _{\infty}$
metric; this is to ensure that metric balls in $K$ will be polyhedral. The ``length metric''
on $K$ is the metric which measures the distance between points $p$ and $q$ as the
infimum of lengths of piecewise linear paths joining $p$ to $q$, where each linear part of
the path lies in, and is measured in, a simplex of $K$. The corresponding metric topology
on $K$ is the same as both the weak topology and the topology inherited from the
Euclidean space $E$.
\vskip 5pt
A locally compact space $A$ is a {\it polyhedron} if it is equipped
with a piecewise linear (PL) structure. This means that the space $A$
has been identified, via a homeomorphism chosen once and for all,
with a rectilinear simplicial complex $K$. The PL {\it structure}
defined by $A$ consists of all rectilinear simplicial complexes $L$ occupying the same
space as $K$, such that $L$ and $K$ have a common rectilinear simplicial
subdivision\footnote{References for polyhedra are \cite{RS}, \cite{Hd}
and Section 3.1 of \cite{Sp}. Here, we use nothing of that subject
beyond basic definitions and elementary properties.}. We say that any
such complex $L$ {\it triangulates} $A$ or {\it is a triangulation of} $A$.
\vskip 5pt
A {\it subpolyhedron} of $A$ is a closed subset $B$ such that some triangulation of $A$ 
has a subcomplex which triangulates $B$. We will often describe a space $A$ as
{\it polyhedral} when we mean that $A$ is a polyhedron in the above sense; and when
$B$ is a
subpolyhedron of $A$ we may simply say that $B$ is ``polyhedral''.
\vskip 5pt
As an extension of this, we may consider an open subset $V$ of the
polyhedron $A$. It is well-known that $V$ has a polyhedral structure
{\it compatible with} that of $A$; this means: $V$ can be triangulated
in such a way that each finite subcomplex is a subpolyhedron of $A$,
and the diameters of simplexes approach zero as the simplexes approach
the frontier of $V$. 

\includegraphics[width=9cm]{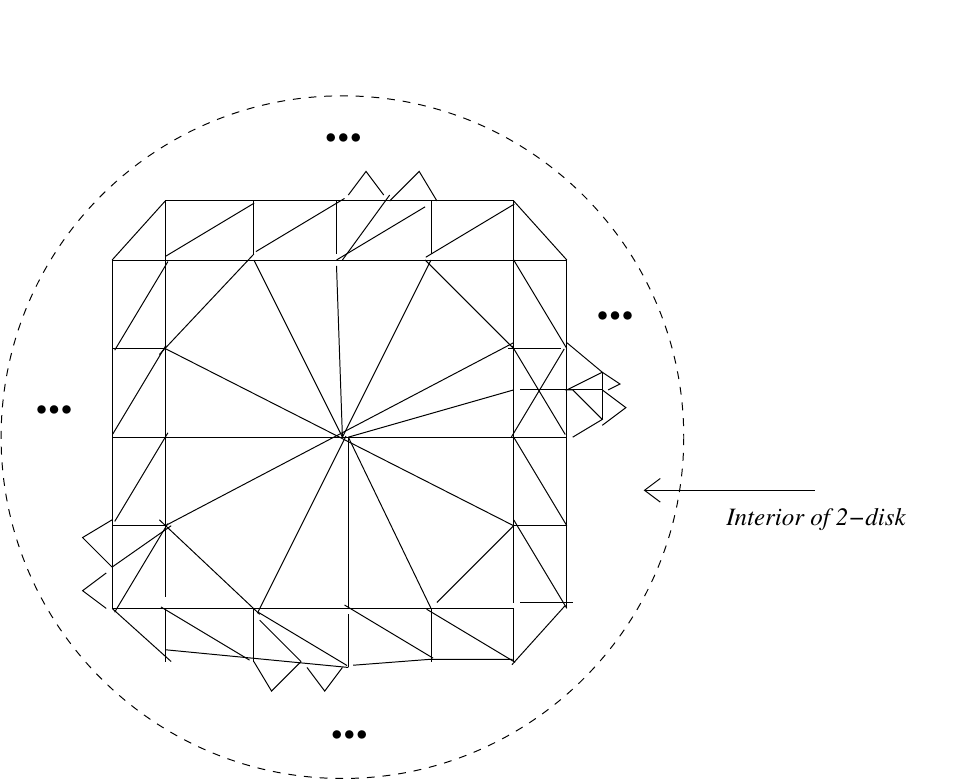}

\vskip 5pt
The polyhedron $A$ has many length metrics, one for each triangulation. In particular,
small
metric balls in polyhedra are cones on their frontiers.  Once a triangulation is chosen, and
hence a length metric, then,  for small enough $\eta >0$, two maps from a compact
domain into $A$ which are distant at most $\eta$ apart pointwise are $\eta
$-homotopic\footnote{Maps $a,
b \colon A\to B$ are $\eta${\it -homotopic} if there is a homotopy between $a$
and $b$ such that for each point $x\in A$ the diameter of the image of
$x\times I$ is $<\eta$.}. This is because when $\eta$ is small and $d(p, q)<\eta$ then that
distance is realized by a unique path which varies continuously with $p$ and $q$.
\vskip 5pt
The following variant on the last sentence will be important for us:
\begin{prop}\label{homotopic}
When $B$ is compact (i.e is triangulated by a finite complex, say $L$)
then there is $\eta _{0}>0$ such that when $\eta <\eta _{0}$ any two maps
from a (not necessarily compact) polyhedron into $B$ which are distant
at most $\eta$ apart pointwise are $\eta$-homotopic.
\end{prop} 
\vskip 5pt
\begin{Def} We call $\eta _0$ the {\it fineness} of $L$.
\end{Def}      
\begin{rems} See \cite{RS} or \cite{Hd} for the following:
\begin{enumerate}[1.]
\item The class of polyhedra is closed under coproducts and finite cartesian products. 
\item If $B$ is a compact polyhedral subset of a finite product $A$ of polyhedra, then
projections
of $B$ to sub-products of $A$ are also compact polyhedral.
\item The join of two compact polyhedra is a compact polyhedron.
\item Smooth manifolds are homeomorphic to polyhedra.
\item The topological join of two locally compact spaces might not be
locally compact, so if the spaces are polyhedra, their join does not satisfy
our definition of ``polyhedron''. However, the join of two abstract
simplicial complexes has an obvious meaning even when they are not locally
finite. Therefore the topological join of two (hence of finitely many)
polyhedra can be triangulated by a simplicial complex which might not
be locally finite. Occasionally, we will allow ourselves to extend the word
``polyhedral'' to include this case.
\end{enumerate}
\end{rems} 
\vskip 5pt

\subsection{Some Consequences of the Additional Properties.}

We now return to the world of geometric $M$-posets\footnote{See Definition
\ref{geometric}.}
and the accompanying notational conventions
introduced in Section \ref{toppo}.
\begin{prop}\label{upcompact}When $r<s$ and $A$ is a compact subset of $X_r$ then
$A^{(s)}$ is a compact subset of $X_s$.
\end{prop}
For this we need notation and a lemma. 

\begin{notation} When $(M,d)$ is a metric space and $S\subseteq cM$, we define
$\bigcup(S) \coloneqq \{m \in M\mid m \in s\text{ for some }s\in S\}$, i.e. $\bigcup(S)$
is the union of the sets in $S$. (The notation $cX$ was defined in Section \ref{geom}.)
\end{notation}

\begin{lemma}\label{hyper} Let $C$ be a compact subset of $c{\mathbb R}^{n}$. Then
$\bigcup(C)$ is a compact subset of $\mathbb{R}^n$.
\end{lemma}
\begin{proof} $C$ is totally bounded in the Hausdorff metric, so $\bigcup(C)$ is a
bounded
subset of ${\mathbb R}^{n}$. If $x$ is a limit point of $\bigcup(C)$ there exists
$(x_{n})$, a
sequence in $\bigcup(C)$, converging to $x$. Let $x_{n}\in c_{n}$ where $c_{n}\in C$.
Without loss of generality, assume
$\{c_{n}\}$ converges to $c\in C$ in the Hausdorff metric. So $x$ must lie in $c$. 
Thus $\bigcup(C)$
is a closed and bounded subset of ${\mathbb R}^{n}$.  
\end{proof}

\begin{proof}[Proof of Proposition \ref{upcompact}] By the definition of a geometric
poset the
set $\pi _{r,s}(A)$ is a compact subset of $cX_s$. The polyhedron $X_s$ can be
regarded as a
subset of a Euclidean space, so Lemma \ref{hyper} implies $A^{(s)}=\bigcup(\pi
_{r,s}(A))$ is
compact. 
\end{proof}   
\vskip 5pt
\begin{prop}\label{poly}
Let $A$ be a compact subset of $X_r$ and let $r\leq s$. If $A$ is polyhedral then
$A^{(s)}$ is
polyhedral.
\end{prop}
\begin{proof}
Define $B \coloneqq (A\times X_{s})\bigcap {\mathcal P}$, where ${\mathcal P}$ is the
order
relation in the $M$-poset ${\mathcal X}$. 
Then $B =\{(a,x)\mid a<x\}\subseteq A\times A^{(s)}$ is polyhedral (being the
intersection of polyhedra). The $X_s$ projection
of $B$ is $\{x\mid a<x \text{ for some }a\in A\}$. This is $A^{(s)}$,
and is compact by Proposition \ref{upcompact}. There are compact polyhedra
$R\subseteq
X_r$ and $S\subseteq X_s$ such that 
$A\times A^{(s)}\subseteq R\times S$. Since $B$ is polyhedral and is
a subset of $R\times S$, $B$ is compact polyhedral. Thus the projection
of $B$, namely $A^{(s)}$, is polyhedral.
\end{proof}
\vskip 5pt

\subsection {Tools from geometric topology}\label{tools}
Here we state three theorems, well-known in geometric topology but perhaps less
well-known in
other mathematical communities.
\vskip 5pt
Recall that a space $Z$ is {\it locally $n$-connected} if for each $z\in Z$ and each
neighborhood
$U$ of $z$ there is a neighborhood $V$ of $z$ such that, for all $k\leq n$, every map
from the
$k$-sphere $S^k$ into $V$ extends to a map of the $(k+1)$-ball $B^{k+1}$ into $U$.
For us,
the important point is that polyhedra are locally $n$-connected for all $n$. 
\vskip 5pt
\begin{Def} When $K$ is a simplicial complex and $L$ is a subcomplex containing all
the vertices of
$K$, a
{\it partial realization} of $K$ in a space $Z$ relative to a cover $\mathcal A$ of $Z$ is a
map $\psi :L\to Z$ such that, for every simplex $\sigma $ of $K$,  there is a set $A\in
{\mathcal
A}$ with $\psi (\sigma\bigcap L)\subseteq A$. When $L=K$ this is a {\it full
realization}.
\end{Def}
\vskip 5pt
\begin{thm}\label{hu2} Let $Z$ be a locally $(n-1)$-connected metrizable space.  Every
open
cover $\mathcal A$ of $Z$ has an open refinement $\mathcal B$ such that every partial
realization $j_{0}:L\to Z$ relative to $\mathcal B$ extends to a full realization $j:K\to Z$
relative to $\mathcal A$. 
\end{thm}
\begin{proof}
This is Theorem V.4.1 of \cite{Hu}.
\end{proof}
\vskip 5pt
Recall that a map is {\it proper} if the pre-image of every compact subset of its codomain
is
compact. A proper map $q:A\to B$ between locally compact polyhedra is {\it cell-like} if 
$q^{-1}(b)$ has trivial shape\footnote{See Section \ref{geom}.} for every $b\in B$.
\vskip 5pt
\begin{thm}\label{lacher} Let $A$ and $B$ be locally compact metric
spaces, $K$ a finite-dimensional, locally finite simplicial complex, $L$
a subcomplex of K, $\theta :A\to B$ a (proper) cell-like map, $\varphi :K\to
B$ a proper map, $\psi :L\to A$ a proper map such that $\theta \circ \psi =
\varphi\mid L$, and $\epsilon :B\to (0,\infty)$ a function. Then there
exists a proper map ${\widebar \varphi}:K\to A$ such that 
$d(\theta \circ {\widebar \varphi}(x), \varphi (x))\leq \epsilon (\varphi (x))$ for every
$x\in K$.
\end{thm}
\begin{proof} This is Lemma 2.3 of \cite{L}.
\end{proof}
\vskip 5pt
\begin{thm}\label{lacher2} A cell-like map between locally compact polyhedra is a
proper
homotopy equivalence.
\end{thm}
\begin{proof} This is Theorem 1.5 of \cite{L}.
\end{proof}
\vskip 5pt

\section{The special case ${\mathcal R}=[2]$}\label{2case}
\begin{notation} In this section, when $A\subseteq X_{1}$ we write $A'$ rather than
$A^{(2)}$.  
\end{notation}
\vskip 5pt
Nearly all the ideas needed for our proof of Theorem \ref{main} occur
when dealing with two special cases, ${\mathcal R}=[2]$ and ${\mathcal R}=[3]$, where
notation is less cluttered. We treat ${\mathcal R}=[2]$ here, and ${\mathcal R}=[3]$ in
the
following section.
\vskip 5pt
When ${\mathcal R}=[2]$, the geometric $M$-poset is ${\mathcal X}=X_{1}\coprod
X_{2}$;
$\mu (X_{1})=1$ and $\mu(X_{2})=2$. It is assumed throughout that ${\mathcal X}$ is
a 
geometric poset. 
\vskip 5pt
In this case, the join $X_{1}*X_{2}$ consists of a copy of $X_{1}$, a copy of $X_{2}$,
and for
each $x\in X_{1}$ and $y\in
X_{2}$ a linearly parametrized path $\omega$ with $\omega (0)=x$ and $\omega
(1)=y$. We call such a path (or its image) a {\it segment} from $x$
to $y$. The point $\omega (t)$ on a segment from $x$ to $y$ is often
more conveniently denoted by $(1-t )x+ty$ where $0\leq t\leq1$. The
map $q\colon X_{1}\times X_{2}\times I\to X_{1}*X_{2}$ taking $(x,y,t)$ to
$(1-t)x+ty$
defines the
quotient topology on $X_{1}*X_{2}$.
\vskip 5pt

\subsection{Weak contractibility of basis elements}\label{2.2} 
Recall that throughout we use the Convenient Basis of Section \ref{convenient}. Basis
elements
lying in $X_2$ are contractible. Here we prove that when $U$ is a (contractible)
basis element in $X_1$ then $\upset U$ is weakly contractible; i.e.

\begin{thm}\label{weakly}For all $n$, $\pi _{n}({\upset U}, U)$ is trivial.
\end{thm}

{\it Notation:} ${\overline U}$ denotes the closure of $U$ in $X_1$.
\vskip 5pt
Theorem \ref{weakly} is a consequence of the  following variant:

\begin{thm}\label{weakly closed} For a basis element $U$, ${\upset {\overline U}}$ is
weakly
contractible.
\end{thm} 
\noindent {\it Proof that Theorem \ref{weakly closed} implies Theorem \ref{weakly}:}
Each basis element $U$ contains basis elements $V_k$ such that ${\overline
V_k}\subseteq V_{k+1}$, and $\bigcup_{k} {\overline V_{k}}=U$.  It follows that 
$${\upset U}=\bigcup_{k}{\upset {\overline V_{k}}}=\bigcup_{k}{\upset {V_{k}}}.$$ 
A singular sphere $\varphi $ in $\upset U$ has compact image, so it lies in some 
${\upset {V_{k}}}$. Thus, by Theorem \ref{weakly closed}, $\varphi $ is homotopically
trivial in $\upset U$. \qed
\vskip 5pt
The rest of this subsection is devoted to proving Theorem \ref{weakly closed}.
\begin{Def} Let $Z$ be a space, and let $A\subseteq X_1$. A map 
$s' \colon Z\to {\uparrow A}$ {\it lies over} a map $s \colon Z\to A$ if 
$s(z)\leq s'(z)$ for every $z\in Z$.
\end{Def}
\vskip 5pt
Fundamental to understanding homotopies in this context is the following
\vskip 5pt
\begin{prop}\label{cover} 
If $s'$ lies over $s$ and $C:=(s')^{-1}(A)$ then, as maps $Z\to \upset A$, $s'$ and $s$
are homotopic rel $C$. 
\end{prop}
\begin{proof}
Define $F \colon Z\times I\to \upset A$ by 
\begin{enumerate}[(i)]
\item $F=s'$ on $Z\times [0,1)$;
\item $F=s$ on $Z\times \{1\}$.
\end{enumerate}
We check continuity of $F$. When $W$ is open in $A'$ then
$F^{-1}(W)=s'^{-1}(W)\times [0,1)$ which is open in $Z\times I$. Whenever $V$ is
open in $A$ we have 
\begin{align*}
F^{-1}(V\cup V') & = F^{-1}(V)\cup F^{-1}(V')\\
& = [s^{-1}(V)\times \{1\}]\cup [s'^{-1}(V')\times [0,1)]
\end{align*}
Since $s^{-1}(V)\subseteq s'^{-1}(V')$ we conclude that $F^{-1}(\upset V)$ is open in
$Z\times I$. Clearly, the homotopy is stationary on $C\times I$.
\end{proof}
\vskip 5pt
{\it Proof of Theorem} \ref{weakly closed}:
We abbreviate the $n$-ball $B^n$ to $B$. Recall that $U$ is a basis element in $X_1$. 
We consider an arbitrary map $$g \colon (B,\del B)\to (\upset {\overline U}, {\overline
U}).$$
\noindent We write $C \coloneqq g^{-1}({\overline U})$. This $C$ is a compact
subset of $B$ containing $\del B$, and ${\overline U}$ is closed in ${\upset {\overline
U}}$. 
\vskip 5pt
{\bf Overall Strategy for the proof:} We will show that there are maps $h'$ and $h$ from
$B$ to $\upset {\overline U}$ satisfying:
\begin{enumerate}[(i)]
\item $h'=h=g$ on $C$;
\item $h'$ maps $B-C$ into ${\overline U}'$;
\item $h$ maps $B$ into ${\overline U}$;
\item $h'$ lies over $h$;
\item $h'$ is homotopic (in $\upset {\overline U}$) to $g$ rel $C$.
\end{enumerate}
\noindent Proposition \ref{cover} will then imply that $h'$ is homotopic to $h$ rel $C$
and hence that $g$ represents the trivial element of $\pi _{n}(\upset {\overline U},
{\overline
U})$.
\vskip 5pt
We need some preparatory lemmas\footnote{Lemma \ref{relation} and Lemma
\ref{uly} are used in proving Lemma \ref{diam} and Lemma \ref{continuous}. Lemma \ref{yznearx} is for
Lemma \ref{uly}.}:
\vskip 5pt
\begin{lemma}\label{relation} If $\{V_m\}_{m=1}^\infty$ is a sequence of
neighborhoods of
$x\in X_{1}$ such that $\bigcap\limits_{m=1}^\infty V_m=\{x\}$, then 
$\bigcap\limits_{m=1}^\infty V_{m}'=x'$. 
\end{lemma}
\begin{proof}
It is clear that $x' \subseteq \bigcap\limits_{m=1}^\infty V_{m}'$.
To show containment in the other direction, let $y \in \bigcap\limits_{m=1}^\infty
V_{m}'$.
By Property A3,  for each $m$, there exists $x_m \in V_m$ such that $x_{m}<y$ in 
$\mathcal{X}$. Since $\bigcap\limits_{m=1}^\infty V_{m}=\{x\}$, $x_m$ converges to
$x$ as $m\to \infty$. Thus, $(x_{m},y)$ converges to $(x,y)$.
Since $\{(x_m,y)\} \subseteq {\mathcal P}$ and the order relation ${\mathcal P}$ is
closed in
$\mathcal{X}^2$, $(x,y)\in {\mathcal P}$.
\end{proof}
\vskip 5pt
\begin{lemma}\label{yznearx} Let $x\in {\overline U}$. Given $\epsilon
>0$, there exists $\eta (x)>0$ such that whenever\footnote{$B_{\delta}$
denotes the ball of radius $\delta$; $N_{\epsilon}$ denotes the $\epsilon
$-neighborhood.} $y$ and $z$ lie in $B_{\eta}(x)$ then $z'\subseteq
N_{\epsilon}(y')$.
\end{lemma}
\begin{proof} This is an immediate consequence of Property A2.  
\end{proof} 
\vskip 5pt
\begin{lemma}\label{uly} Given $\epsilon >0$ there exists $\delta >0$
such that whenever $p\in C$ and $q\in B_{\delta}(p)-C$ then $g(q)\in
N_{\epsilon}(g(p)')$.  
\end{lemma}
\begin{proof} We write ${\mathcal C} \coloneqq \{\upset B_{\eta (x)}(x)\mid x\in
{\overline
U}\}$, where
$\eta(x)$ is as in Lemma \ref{yznearx}. Then $g^{-1}({\mathcal C})$ is an open cover of
the
compact set
$C$. Let $\delta $ be a Lebesgue number for this cover. For some $x\in {\overline U}$,
$g(B_{\delta}(p))\subseteq \upset B_{\eta (x)}(x)$. This means that
for some $w\in B_{\eta (x)}(x)$, $g(q)\in w'$. By Lemma \ref{yznearx},
$g(q)\in N_{\epsilon}(g(p)')$.  
\end{proof}
\vskip 5pt
Now we can proceed with the proof. While it looks complicated, it mainly consists of
$\epsilon-\delta$ arguments.
\vskip 5pt

Let $Q \coloneqq({\overline U}\times {\overline U}')\bigcap {\mathcal
P}$. By Propositions \ref{upcompact} and \ref{poly}, $Q$ is a compact
polyhedron. We choose triangulations $L$ for ${\overline U}$ and $M$
for $ {\overline U'}$. Then $L\times M$ has a subdivision triangulating
$Q$. Each of $L$ and $M$ carries its length metric, and the metric on
${\overline U}\times {\overline U}'$ is taken to be the sum of those
separate metrics. Let $\eta _{0}>0$ be a number less than the fineness
of both length metrics (see Proposition \ref{homotopic}).
\vskip 5pt


For each integer $i\geq 0$ we choose an open cover ${\mathcal A}_{i}$ of $Q$ whose
members have diameter $<\frac{1}{i}$ ($\frac{1}{0}=\infty$) and let ${\mathcal B}_i$
be the refinement given
by Theorem \ref{hu2}. Let $4\lambda _i$ be a Lebesgue number for ${\mathcal B}_i$,
where $\lambda _i$ decreases to $0$ as $i\to \infty$. It will be important later that
$\lambda _{0}$ be chosen to be less than $\eta _{0}$.

For $x\in B-C$ we choose $r(x)\in C$ such that\footnote{It is not claimed that $r$ is
continuous. We will only be interested in its restriction to the (discrete) $0$-skeleton of a
chosen triangulation of $B-C$.} $d(x,r(x))=d(x,C)$. We write $k_{0}(x)=g\circ r(x) \in
{\overline U}$. 
\vskip 5pt
By Property A1, the set $k_{0}(x)'$ is compact and non-empty.
By Lemma \ref{uly}, there is a sequence $(\gamma_i)$, decreasing to $0$,
with the following property: when $x\in B-C$ satisfies $\gamma_{i-1}\leq
d(x,C)\leq \gamma _{i}$ there exists a point $k_{0}'(x)\in k_{0}(x)'$ such
that $d(k_{0}'(x)), g(x))\leq \lambda _i$. As $i$ varies, this defines a
function $k_{0}':B-C\to {\overline U}$. Note that $k_{0}(x)<k_{0}'(x)$
in the poset. We write $j_{0} \coloneqq(k_{0},k_{0}'):B-C\to Q$.
\includegraphics[width=9cm]{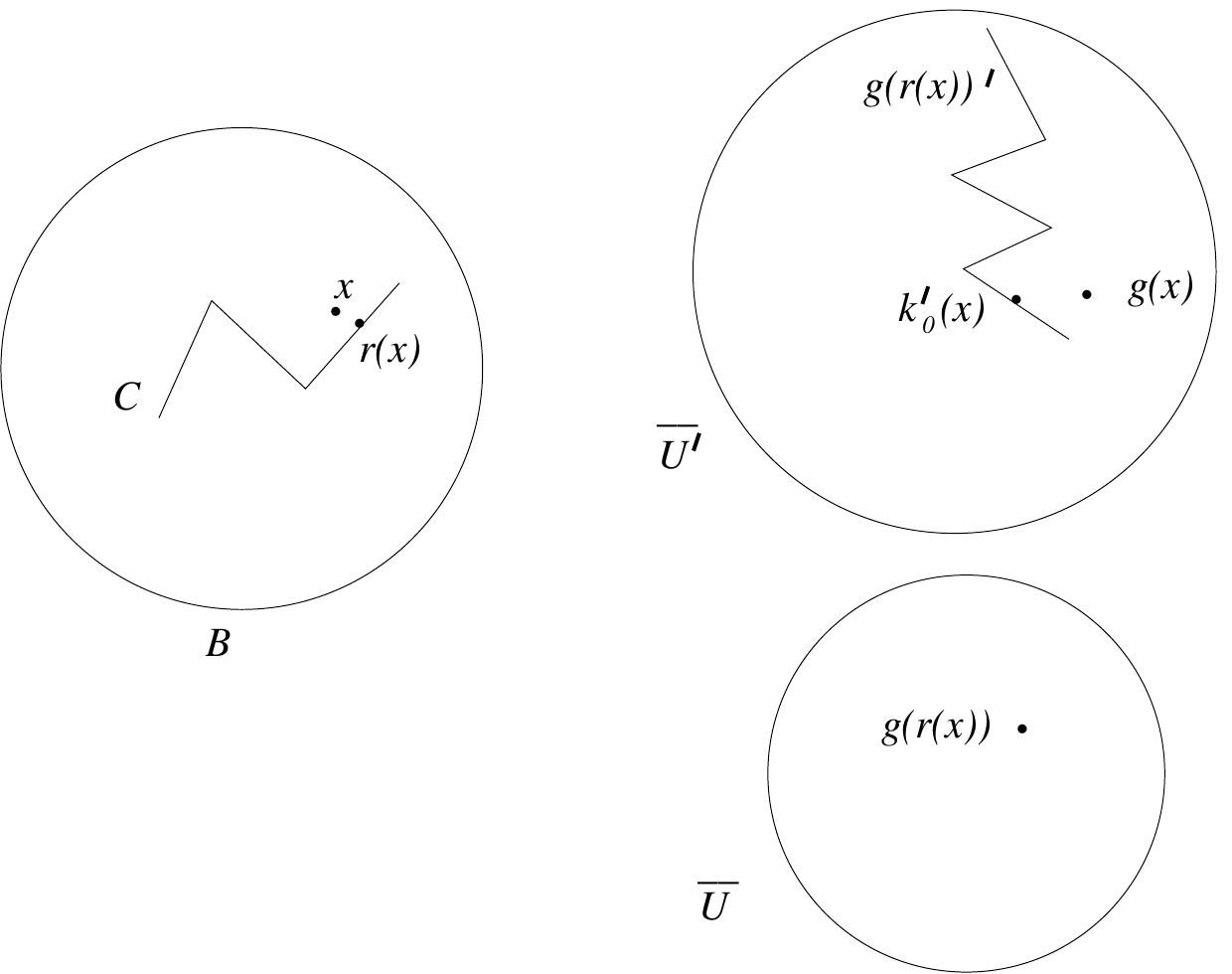}
\vskip 5pt 
We define $I_{i} \coloneqq\{x\in B-C\mid \gamma _{i+1}\leq d(x,C)\leq
\gamma_{i}\}$, and $N_{i}\coloneq I_{i-1}\bigcup  I_{i}\bigcup
I_{i+1}$. These sets are compact but might not be polyhedral. 
The map
$g$ is uniformly continuous on $N_i$, so there is a sequence $(\theta _{i})$, decreasing
and converging to $0$, such
that when $S$ is a subset of $N_i$ of diameter $<\theta _{i}$ then $g(S)$
has diameter $<\lambda _{i}$. 
\vskip 5pt
These choices give us control over how the $0$-skeleton of a suitable triangulation of
$B-C$ is mapped into $Q$, made precise in the following:

\begin{lemma}\label{diam} Let $K$ be a triangulation of $B-C$ such that
each $I_i$ is covered by a finite subcomplex $K_i$ lying in $N_i$,
and every simplex of $K_i$ that meets $I_{i}$ has diameter $<\theta
_i$. There is an increasing sequence of positive integers
     $\ell (i)\geq i$ such that the complexes $K_{\ell (i)}$ are pairwise
     disjoint, and when the simplex $\sigma$ of $K_{\ell (i)}$ meets
     $I_{\ell (i)}$ then 
     \begin{enumerate}[\rm (i)]
          \item $k_{0}(\sigma ^{0})$ has diameter $<\lambda _{i}$
          \item $k_{0}'(\sigma ^{0})$ has diameter $<3\lambda
          _{i}$.
     \end{enumerate}
\end{lemma}

\begin{proof} When $\ell >0$ is an integer and $\sigma $ is a simplex of $K_{\ell +1}$
that meets $I_{\ell +1}$, then $\sigma ^{0}$ has diameter $\leq \theta _{\ell+1}$. 

     For any $u$ and $v$ in $\sigma ^{0}$, we have 
     $d(u, v)\leq \theta _{\ell +1}$, $d(u, r(u))\leq \gamma _{\ell}$, and 
     $d(v, r(v))\leq \gamma _{\ell}$. So 
     $d(r(u), r(v))\leq 2\gamma _{\ell}+\theta _{\ell +1}$. This is an upper bound for
the diameter of $r(\sigma ^{0})$.  

     The restriction $g\mid :C\to {\overline U}$ is uniformly continuous, so there is a
sequence $(\rho _{i})$, decreasing and converging to $0$, such that when 
     $r(\sigma ^{0})$ has diameter $<\rho _i$ then $g\circ r(\sigma ^{0})$ (i.e.
$k_{0}(\sigma ^{0}))$ has diameter $<\lambda_{i} $.  

     So, given $i$, when $\ell $ is large enough that $2\gamma _{\ell}+\theta _{\ell
+1}<\rho _{i}$ then 
     $k_{0}(\sigma ^{0})$ has diameter $\leq \lambda _{i}$. For each $i$, pick $\ell
=\ell (i)$ to be large enough in the above sense, and such that the complexes $K_{\ell
(i)}$ are pairwise disjoint, while the sequence $(\ell (i))$ is increasing. Then (i) is
satisfied. 
     \vskip 5pt
     When $\sigma $ meets $I_i$, the above proof gives 
     $d(u,
     r(u))\leq \gamma _{i}$, and $d(v, r(v))\leq \gamma _{i}$,
     hence $d(k_{0}'(u), g(u))\leq \lambda _i$ and $d(k_{0}'(v),
     g(v))\leq \lambda _i$. Since $d(g(u), g(v))\leq \lambda _i$,
     we see that $k_{0}'(\sigma ^{0})$ has diameter $<3\lambda _{i}$,
     so (ii) is satisfied.
\end{proof}
\vskip 5pt 
\includegraphics[width=9cm]{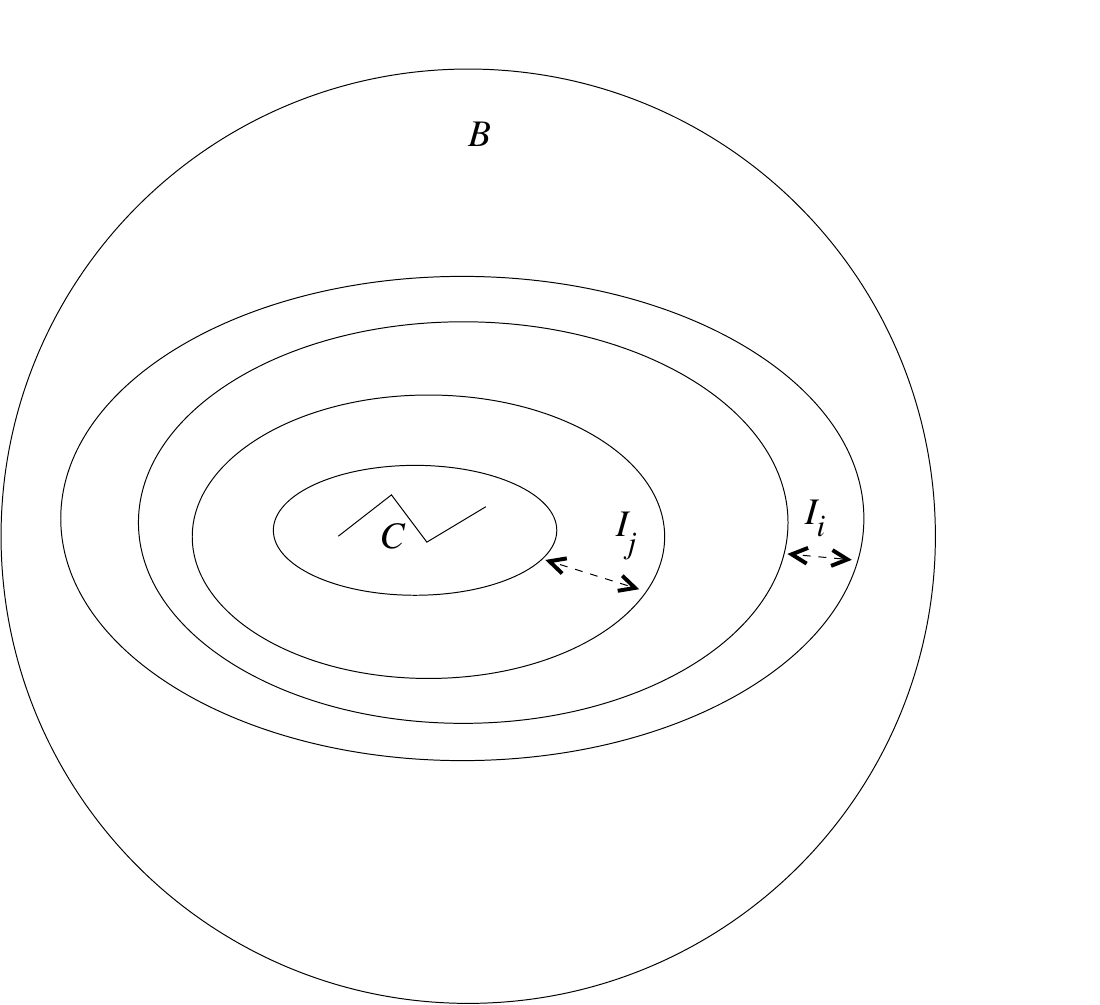}
The ``annuli" $K_{\ell (i)}$ in $B-C$ bear down on $C$ but there will be gaps between
them. The larger complex 
$J_{\ell (i)} \coloneqq\bigcup _{\ell (i)\leq m\leq \ell (i+1)}K_{m}$ contains the $i$th
gap. Each $J_{\ell (i)}$ is a subpolyderon of $B-C$. We may therefore assume that it is
triangulated as a subcomplex of $K$. Since the sequences 
$(\gamma _{i}), (\theta _{i}),$ and $(\rho _{i})$ are all decreasing, we can extend the
lemma as follows:
\begin{addendum}\label{add} The inequalities {\rm (i)} and {\rm (ii)} in Lemma
\ref{diam} continue to hold when $\sigma $ is a simplex of $J_{\ell (i)}$.
\hfill$\square$
\end{addendum}
We consider the function $j_{0}:=(k_{0},k_{0}'):B-C\to Q$ but only on the $0$-skeleton
of $J$ 
(so continuity is not a problem). By Addendum \ref{add}, for every
simplex $\sigma $ of $J_{\ell (i)}$ the set $j_{0}(\sigma ^{0})$
has diameter $<4\lambda _i$ and thus lies in an element of the cover
${\mathcal B}_i$ of $Q$. 

The sequence $\{\ell (i)\}$ starts with $i=0$ (recall that $\frac{1}{0}=\infty$).
The definition of the covers ${\mathcal A}_{i}$ allows us to assume that ${\mathcal
A}_{0}$
is the singleton $\{Q\}$. The refinement ${\mathcal B}_{0}$ has Lebesgue number 
$4\lambda _{0}$.
\vskip 5pt
Let $J \coloneqq\bigcup _{i\geq 0}J_{\ell (i)}$, a subcomplex of
$K$. We wish to extend $j_{0}\mid :J^{0}\to Q$ to a map $j:J\to Q$. We do
this by induction. On $J_{\ell (0)}$, $j_{0}$ maps the $0$-skeleton of
each simplex to a set of diameter $<4\lambda _{0}$. Theorem \ref{hu2}
gives an extension $j$ on $J_{\ell (0)}$. Next, we consider $j_{0}^{+}$ on
$J_{\ell (1)}$, where $j_{0}^{+}=j_{0}$ on the vertices of $J_{\ell (1)}$
and $j_{0}^{+}$ agrees with the (previously defined) $j$ on $J_{\ell (0)}\bigcap J_{\ell
(1)}$. We subdivide this intersection so that the
$j$-image of each simplex in it has diameter $<4\lambda _{1}$, and we
extend this subdivision to  $J_{\ell (0)}\bigcup J_{\ell (1)}$ without
adding further vertices. At this point we have a partial realization\footnote{See Section
\ref{tools}.} of 
$J_{\ell (1)}$ in $Q$ relative to ${\mathcal B}_{1}$, and Theorem \ref{hu2}
gives an extension of $j$ to all of $J_{\ell (1)}$ extending the map $j$
previously defined on $J_{\ell (0)}$. The $j$-image of each simplex in
$J_{\ell(1)}$ lies in an element of ${\mathcal B}_{1}$.  We continue
in this way so that on $J_{\ell (i)}-\text {\rm int}J_{\ell(i-1)}$ the
$j$-image of the $0$-skeleton of each simplex has diameter  $<4\lambda _{i}$, each
time 
using Theorem \ref{hu2}.
\vskip 5pt
We point out that $J$ does not include all of $B-C$, but $N\coloneqq J\bigcup C$ is a
polyhedral neighborhood of $C$ in $B$. The map $j=(k,k')$ has components $k:J\to
{\overline U}$ and $k':J\to {\overline U}'$. We extend both of these
maps\footnote{Previously, we defined $k_0$ and $k'_0$ on all of $K$. We now confine
the domains of $k_0$ and $k'_0$ to $J$. In
what follows we need to be free to extend $j$, i.e. $(k, k')$, to all
of $K\bigcup C$ in a different way.} (without change of label) to agree with
$g$ on $C$. Thus $k$ and $k'$ now have domain $N$. 
\vskip 5pt
\begin{lemma}\label{continuous} The extended functions $k:N\to {\overline U}$ and 
$k':N\to {\overline U}\bigcup {\overline U'}$ are continuous.
\end{lemma}
\begin{proof}The only issue is continuity at any point $p\in C$. Let ${\upset T}$ be a
basic
(open) neighborhood of $g(p)$. Since $g$ is continuous $g^{-1}({\uparrow T})$ is an
open
neighborhood of $p$. It is enough to find a smaller neighborhood, $W$, of $p$ such that
$k'(W)\subseteq T'$ and $k(W)\subseteq T$. 
\vskip 5pt
If $\{{\overline V}_{m}\}$ is a basic sequence of compact neighborhoods
of $g(p)$ in $U$, then $\bigcap \{{\overline V}_{m}'\}=g(p)'$
by Lemma \ref{relation}.                               
By Property A1, $g(p)'$ is a compact non-empty subset of
the open set $T'$, so for sufficiently large $m$ the
compact sets ${\overline V}_{m}'$ and the (closed) frontier of $T '$ are
disjoint. Choose $\epsilon $ so that $2\epsilon =d(g(p)', \text{fr}(T'))$. For this choice of
$\epsilon $, let
$\delta$ be as in Lemma \ref{uly}. The following shows that $B_{\delta}(p)$ is the
required
neighborhood $W$.
\vskip 5pt
When $x$ is a point of $K$ such that $d(x,p)<\delta$ then $d(x,r(x))<\delta$. By
Property A1,
$g(r(x))'$ is compact and non-empty. 
Since $d(x,r(x))<\delta$, $d(g(x),g(r(x))')<\epsilon$,
hence $d(g(x),k'(x))<\epsilon$. And since $d(x,p)<\delta$,
$d(g(x),g(p)')<\epsilon$. Hence $d(k'(x), g(p)')<2\epsilon$. This implies
$k'(x)\in T'$.
\vskip 5pt
Continuity of $k$ at $p$ follows from the fact that $g$ is continuous
at $p$: i.e. given $\epsilon >0$ let $\delta >0$ be such that $g$ maps the
$\delta $-neighborhood of $p$ into the $\epsilon $-neighborhood of
$g(p)$ in $U$. If $d(p,x)<\frac{\delta }{2}$ then $d(p,r(x))<\delta $,
so $d(g(p), g(r(x)))<\epsilon$; i.e. $d(k(p),k(x))<\epsilon$. Thus
$k$ is continuous at $p$.
\end{proof}
\vskip 5pt
The required maps (see Strategy, above) $h$ and $h'$ are defined to agree with $k$ and
$k'$ respectively on $N$. Then $h=h'=g$ on $C$, and $h'$ lies over $h$ (on $N$).
\vskip 5pt

Because the $\lambda _i$ are less than the fineness of our triangulation of 
$\overline U$, Proposition \ref{homotopic} implies that the maps $h'|N-C$ and $g|N-C$
are homotopic. Moreover, the homotopy gets smaller and smaller as $i$ increases, and so
it extends to $g\times \text{\rm identity}$ on $C$. In other words, $h'$ and $g$ are
homotopic on all of $N$.   

\vskip 5pt
It remains to extend $h$ and $ h'$ to the rest of $K$, maintaining the properties (i)-(v) of
the Strategy. 
\vskip 5pt
\includegraphics[width=11cm]{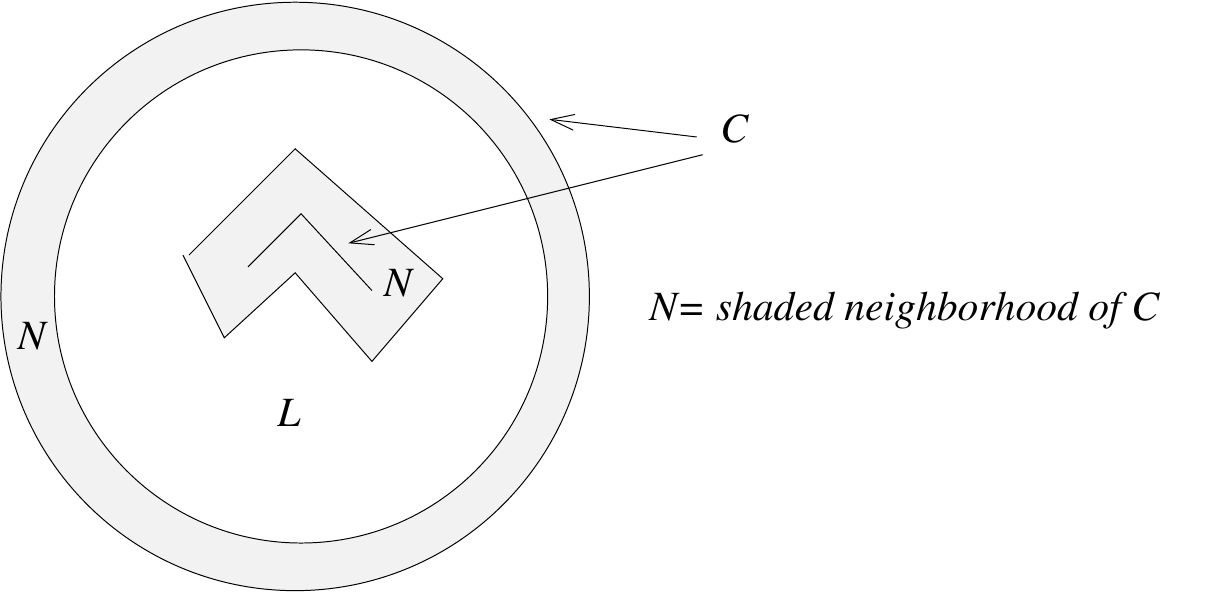}
\vskip 5pt
Recall that $Q \coloneqq({\overline U}\times {\overline U}')\bigcap {\mathcal P}$. The
projection map
$\pi :({\overline U}\times {\overline U}')\to {\overline U}'$ is easily seen to restrict to a
proper
map $Q\to {\overline U}'$,
and, by Property A3 (of the Additional Properties) this map is cell-like.
\vskip 5pt
Let $L$ be the finite subcomplex of $K$ covering the closure of the complement of $N$
in
$B$; i.e. $L\bigcup J=K$. Then $L$ is disjoint from the interior of the
neighborhood\footnote{Recall our convention of using the letter $L$ for both the
simplicial complex and its geometric realization.} $N$. We write ${\overset \bullet
L}=L\bigcap N$, the subcomplex of $L$ such that 
${\overset \bullet L}$ is the frontier of $N$ in $B$. We
have a map $(L\times 0)\bigcup ({\overset \bullet L}\times I)\to {\overline U}'$
agreeing with $g$ on $L\times 0$ and being a homotopy between $g|{\overset
\bullet L}$ and $h'|{\overset \bullet L}$ on ${\overset \bullet L}\times
I$. The Homotopy Extension Property gives a map ${\widebar g}:L\to {\overline U}'$
homotopic to $g$ and agreeing with $h'$ on ${\overset \bullet
L}$. Now, $j|:{\overset \bullet L}\to Q$ satisfies $\pi \circ j|=h'|:{\overset \bullet L}\to
{\overline U}'$. Theorem \ref{lacher} then allows us to extend $j$ continuously to all of
$L$. 

$$
\xymatrix{
{\overset \bullet L} \ar[d]_{\text{inclusion}} \ar[r]^{j\mid} &Q\ar[d]^{\pi\mid}\\
L \ar[r]^{\widebar g} \ar[ur]^{j} &{\overline U}'}
$$

Where the domains overlap,
this $j$ agrees with that previously defined on $N$, so, gluing the
two maps together, we have $j:K\to Q$. The components of $j$ are the required maps
$h$ and
$h'$. They extend the previously defined $h$ and $h'$. By Theorem \ref{lacher}, the
component
$h' (=\pi \circ j)$ can be as close as we please to $\widebar g$, and hence is homotopic to
it.
Thus, by Proposition \ref{cover}, $g$ is homotopic, rel $\del B$, to a map $h$ 
whose image misses ${\overline U}'$. 
\vskip 5pt
This completes the proof of Theorem \ref{weakly closed}, hence also of Theorem
\ref{weakly}.

\subsection{Pre-images of compact polyhedral sets} Here we prove Theorem
\ref{inverse}, relating $\Delta ({\upset A})$ to $f^{-1}(\upset A)$
when $A$ is a compact polyhedral subset of $X_1$.

\vskip 5pt
Let $d$ be a length metric on $X_1$. Define $\varphi \colon X_{1}\times (0,1]\to I$ by
$\varphi
(x,u)=\text{min}\left\{\frac{d(x,A)}{u},1\right\}$. Writing $\varphi _{u}$ for 
$\varphi (\cdot,u)$ we note

\begin{enumerate}[(i)]
\item $\varphi _{u}=0$ on $A$;
\item $\varphi _{u}=1$ on $X_{1}-N_{u}(A)$;
\item $0<\varphi _{u}<1$ elsewhere.
\end{enumerate}
\vskip 5pt
Let $B$ be a compact polyhedral subset of $X_2$.
For each $u \in (0, 1]$, define
$\psi _{u} \colon X_{1}\times B\times I\to X_{1}\times B\times I$ by 

\[ \psi _{u}(x,y,t)=\begin{cases} 
      (x,y,\varphi _{u}(x)) & \text{ if } t\leq \varphi _{u}(x) \\
      (x,y,t) & \text{ if } \varphi _{u}(x)\leq t.
   \end{cases}
\]

\noindent We note that $\psi _{u} $ fixes $A\times B\times I$ and that
$\psi _{u} (x,y,t)=(x,y,1)$ when $x\in X_{1}-N_{u}(A)$ and $y\in B$.
\vskip 5pt
We also note the homotopy
$h_{u} \colon X_{1}\times B\times I\times [0,1]\to X_{1}\times B\times I$
between $\psi _{u}$ and the identity map defined by

\[ h_{u}(x,y,t,v)=\begin{cases} 
      (x,y,v\varphi _{u}(x)) & \text{ if } t\leq v\varphi _{u}(x) \\
      (x,y,t) & \text{ if } v\varphi _{u}(x)\leq t
   \end{cases}
\]

\noindent and that the image of $\psi _{u}$ is the subset
$\{(x,y,t) \in X_{1}\times B\times I\mid \varphi _{u}(x)\leq t\leq1\}$.
\vskip 5pt
Factoring by the canonical quotient $X_{1}\times B\times I\to X_{1}\ast
B$ we see that, while $\psi _{u}$ does not induce a well-defined map $X_{1}\ast
B\to X_{1}\ast B$, it does induce a map 
$\Psi _{u} \colon (X_{1}\ast B)-(X_{1}-A)\to X_{1}\ast B$. Similarly, the homotopy
$h_u$
induces a homotopy $H_{u} \colon (X_{1}\ast B)-(X_{1}-A))\times [0,1]\to X_{1}\ast
B$.
\vskip 5pt
The image of $\Psi _{u}$ is the compact set
$$I_{u} \coloneqq \{(1-t)x+ty\in X_{1}\ast B\mid x\in X_{1}, y\in B, t\geq \varphi
_{u}(x)\}$$.
\vskip 5pt
\begin{prop} Given $u$ in $(0,1]$ and a neighborhood $M$ of $N_{u}(A)\ast B$ in
$X_{1}\ast
X_{2}$, we can deform $(X_{1}\ast B)-(X_{1}-A)$ within itself onto the
compact set $I_{u}\subseteq M$ by a deformation that is continuous in the variable $u$.   
\end{prop}
By careful choice of the metric $d$ it can be arranged that for $u=\frac{1}{n}$ the above
functions are PL, and in particular that $I_{\frac{1}{n}}$ is polyhedral. Thus we get:
\begin{prop} Given $u=\frac{1}{n}$ and a neighborhood $M$ of $N_{u}(A)\ast B$ in
$X_{1}\ast X_{2}$, we can deform $(X_{1}\ast B)-(X_{1}-A)$ within itself onto the 
compact polyhedral set $I_{\frac{1}{n}}\subseteq M$.   
\end{prop} 
\vskip 5pt
Since the order relation ${\mathcal P}$ is polyhedral, this proposition can be restricted as
follows, where we
now take $B$ to be  $A'$ (compact and polyhedral by Proposition \ref{poly}):
\vskip 5pt
\begin{cor}\label{polyneigh} Given $u=\frac{1}{n}$ and a neighborhood $M$ of 
$\Delta (\upset A)$ in $\Delta ({\mathcal X})$, $f^{-1}(\upset A)$
can be deformed within itself onto a compact polyhedral set $J_{n}\subseteq M$.   
\end{cor} 
\begin{proof} The set  $f^{-1}(\upset A)$ consists of $\Delta (\upset
A)$ together with line segments ending in $A'$ whose intersections with
$X_{1}$ have been deleted. Here, we write $J_n$ for the relevant part
of  $I_{\frac{1}{n}}$.  \end{proof}
\vskip 5pt
We note that $J_{n+1}$ is a strong deformation retract of $J_n$.
\vskip 5pt
\begin{lemma}\label{delta} The set $\Delta (\upset A)$ is a compact polyhedral subset of
$X_{1}*X_{2}$.
\end{lemma}
\begin{proof}
By Lemmas \ref{upcompact} and \ref{poly}, $A'$ is compact and
polyhedral. Thus $C \coloneqq ((A\times A')\bigcap {\mathcal P})\times I$ is
compact and polyhedral. The projection maps $A\times A'\to A$ and $A\times
A'\to A'$ define PL maps $\pi _{A} \colon ((A\times A')\bigcap {\mathcal P})\times
\{0\}\to
A$ and $\pi _{A'} \colon ((A\times A')\bigcap {\mathcal P})\times \{1\}\to A'$. The
space
$\Delta (\upset A)$ is clearly decomposable as the union of three compact
polyhedra: $C$ and the mapping cylinders of $\pi _{A}$ and $\pi _{A'}$,
where the mapping cylinders are glued to the $0$- and $1$- ends of $C$
in the obvious way. Since mapping cylinders of PL maps between compact polyhedra
are compact polyhedra, and the union of compact polyhedra glued along compact
polyhedral
subsets is again a compact polyhedron, this completes the proof.
\end{proof}
\vskip 5pt
\begin{thm}\label{inverse}
When $A$ is a compact polyhedral subset of $X_1$, the inclusion map $\Delta (\upset A)
\hookrightarrow f^{-1}(\upset A)$ is a homotopy equivalence.
\end{thm}
\begin{proof}
$\Delta (\upset A)$ is the intersection of the nested sequence of polyhedra $J_n$, each of
which
is a strong deformation retract of $f^{-1}(\upset A)$ by Corollary \ref{polyneigh}.  
Since the inclusion 
maps $J_{n+1}\to J_n$ are homotopy equivalences, shape theory (see Section
\ref{appendix} for
details) implies that $\Delta (\upset A)$ is shape equivalent to each of the polyhedra
$J_n$.
Moreover, since $\Delta (\upset A)$ is itself polyhedral by Lemma \ref{delta}, this means
that
$\Delta (\upset A)$ is actually homotopy equivalent to each of the $J_n$, hence also to
$f^{-1}(\upset A)$.  
\end{proof}

\vskip 5pt

\subsection{Contractibility of pre-images of basis elements}\label{2.1}

The order complex is the subset $\Delta ({\mathcal X}) \subseteq X_{1}*X_{2}$  which
consists of $X_1$, $X_2$, and the union of all segments joining $x$ to $y$ such that
$x<y$.
The Comparison Map $f \colon \Delta ({\mathcal X})\to {\mathcal X}$ is defined by 
$f(x)=x$ when $x\in X_{1}$ (i.e. when $t=0$)  and $f((1-t)x+ty)=y$ when $t>0$.

As before, we consider the two kinds of basis elements in the Up topology:
$U\in {\mathcal U}_{2}$ and $\upset U$ where $U \in {\mathcal U}_{1}$. To
apply Theorem \ref{mccord} we show that the pre-image under $f$ of each
of these is weakly contractible.
\vskip 5pt
When $U\in {\mathcal U}_{2}$, the set $f^{-1}(U)$ consists of the contractible set 
$U\subseteq X_{2}\subseteq X_{1}*X_{2}$, together with half-open segments ending in
$U$.
Thus $U$  is a strong deformation retract of $f^{-1}(U)$ which is therefore contractible.
\vskip 5pt
Now we consider the case where $U\in {\mathcal U}_{1}$. It is convenient to first deal
with 
$f^{-1}(\upset {\overline U})$, which Theorem \ref{inverse} tells us is homotopy
equivalent to 
$\Delta (\upset {\overline U})$. 

\begin{notation} We write ${\mathcal C}A$ for the topological cone on the space $A$, 
i.e. $A\times I\slash A\times \{0\}$.
\end{notation}

\begin{prop}\label{cone} $\Delta (\upset {\overline U})$ is homotopy equivalent to
${\mathcal
C}\overline U'$, and is therefore contractible.
\end{prop}
\begin{proof} Let $g\colon\Delta (\upset {\overline U})\to {\mathcal
C}\overline U'$ be the map which takes $\overline U$ to the cone point $p$, is the
identity on $\overline U'$, and maps the point  $(1-t)x+ty$ to the
segment  $(1-t)p+ty$ in the cone. By Lemma \ref{delta} this is a map between compact
polyhedra. Consider $g^{-1}((1-t)p+ty)$. When
$t=0$ this pre-image is the contractible set $\overline U$; when $t=1$
this pre-image is a single point; when $0<t<1$ this pre-image is homeomorphic to $\{x\in
{\overline U}\mid x<y\}$ which has trivial shape (actually, is contractible) by Property
A3. 
Thus $g$ is a cell-like map, and is therefore a homotopy equivalence by Theorem
\ref{lacher2}.
\end{proof}
\vskip 5pt
\begin{cor}\label{contractible} $f^{-1}(\upset {\overline U})$ is contractible.  
\end{cor}
\vskip 5pt
Just as Proposition \ref{weakly closed} implies Theorem \ref{weakly},
this implies that (for basis elements $U$), $f^{-1} (\upset {U})$
is weakly contractible.  Since it is an open subset of a polyhedron,
it is actually contractible.
\vskip 5pt
Theorem \ref{main} for the case ${\mathcal R}=[2]$ now follows by combining
Theorem \ref{weakly closed}, Corollary \ref{contractible} and Theorem \ref{mccord}. 

\section{The special case ${\mathcal R}=[3]$}\label{3case}
\vskip 5pt
\begin{notation} In this section, when $A\subseteq X_{1}$ we write $A'$ rather than
$A^{(2)}$
and $A''$ rather than $A^{(3)}$.  
\end{notation}
\vskip 5pt 
\subsection{Weak contractibility of basis elements}\label{3weak}
Here ${\mathcal X}=X_{1}\coprod X_{2}\coprod X_{3}$ and $U\subseteq X_i$ is a
basis
element. We are to show that $\uparrow U$ is weakly contractible. If $U\subseteq X_3$
this is
trivial, and if $U\subseteq X_2$ this follows from what was done in Section \ref{2case}.
So we
assume $U\subseteq X_1$. 
\vskip 5pt
Once again, it is easier to begin with $\overline U$. We consider an arbitrary map
$$g \colon (B,\del B)\to (\upset {\overline U}, {\overline U}\bigcup {\overline U}').$$ 
\noindent where, once again, $B$ is an abbreviation of $B^n$. We write $C \coloneqq
g^{-1}({\overline U}\bigcup {\overline
U}')$. As before, this is a closed subset of $B$ containing $\del B$. We wish to produce
maps
$h''$
and $h$ from $B$ to $\upset {\overline U}$ satisfying:
\begin{enumerate}[(i)]

\item $h''=h=g$ on $C$;
\item $h''$ maps $B-C$ into ${\overline U}''$;
\item $h$ maps $B$ into ${\overline U}\bigcup {\overline U}'$;
\item $h''$ lies over $h$;
\item $h''$ is homotopic to $g$ rel $C$.
\end{enumerate}
\noindent Proposition \ref{cover} will then imply that $h''$ is homotopic to $h$ rel $C$
and hence that the map $g$ represents the trivial element of $\pi _{n}(\upset {\overline
U},
{\overline U}\bigcup{\overline U'})$; hence that group is trivial. By Proposition
\ref{weakly
closed} it will
follow that $\pi _{n}(\upset {\overline U}, {\overline U})$ is trivial.
As before, we can deduce the same when ${\overline U}$ is replaced by $U$.  
\vskip 5pt
The argument here is more complicated than in Section \ref{2case} because the obvious
replacement for ${\overline U}\times {\overline U}'$ in the definition of $Q$, namely
$({\overline U}\bigcup {\overline U}')\times {\overline U}''$, is not polyhedral (with the
Up
topology on ${\overline U}\bigcup {\overline U}'$). For this reason we must construct
the map
$h''$ in two stages, one for ${\overline U}$, the other for ${\overline U}'$.  
\vskip 5pt
Let $Q_{1} \coloneqq({\overline U}\times {\overline U}'')\bigcap {\mathcal P}$. As
with $Q$, this is a
compact polyhedron. Let ${\mathcal A}_{i}$ be an
open cover of $Q_{1}$
by sets of diameter $<\frac{1}{i}$, and let ${\mathcal B}_i$ be the
refinement given by Theorem \ref{hu2}. Let $4\lambda _i$ be a Lebesgue
number for ${\mathcal B}_i$.
\vskip 5pt
We write $C_{1} \coloneqq g^{-1}({\overline U})$; this is a closed subset of 
$C$. Define $Y \coloneqq\{x\in B-C\mid d(x,C)=d(x,C_{1})\}$, a closed subset of
$B-C$.
\vskip 5pt
We now proceed as in Section \ref{2.2} replacing $B-C$ by $Y$ and $C$ by $C_1$. The
triangulation $K$ and its subcomplex $J$ are replaced by a triangulation $K_1$
of $Y$ and its subcomplex $J_1$. We define $N_{1} \coloneqq J_{1}\bigcup C_{1}$, a
polyhedral
neighborhood of
$C_1$ in $Y$. We get a map $j_{1} \coloneqq(k_{1},k_{1}'')$; with components  
$k_{1}:J_{1}\to {\overline U}$ and $k_{1}'':J_{1}\to {\overline U}''$. Both of these
maps are
extended to agree with $g$ on $C_1$. As before, the extended function $j_{1}:N_{1}\to
Q_{1}$
is continuous.  
\vskip 5pt
We first define the maps $h$ and $h''$ on $N_1$, where they agree with $k_{1}$ and
$k_{1}''$
respectively. Then $h''$ lies over $h$ and $h=h''=g$ on $C_1$. As in Section \ref{2.2}
we may
assume $h''$ and $g|N_1$ are homotopic rel $C_1$. The extension of $h$ and $ h''$ to
the rest of
$K_1$ maintaining the properties (i)-(v) is achieved as in that section. We further extend
$h$ and
$h''$ to agree with $g$ on all of $C$. At this point, the Homotopy Extension Property
implies
that $g:B\to \upset {\overline U}$ is homotopic rel $C$ to a map $g_{1}:B\to \upset
{\overline
U}$, where $g_{1}$ agrees with $h''$ on $Y\bigcup C$.
\vskip 5pt
Now comes the part that is different from Section \ref{2.2}. We wish to alter $g_1$ rel 
$Y\bigcup C$ to have the desired covering property on the rest of $B-C$. The set 
$Z \coloneqq B-(C_{1}\bigcup Y)$ is an open polyhedral subset of $B$ whose frontier is
the compact set
${\overset \bullet Z} \coloneqq\text{\rm fr}_{B}(Z)$.
\vskip 5pt
\includegraphics[width=10cm]{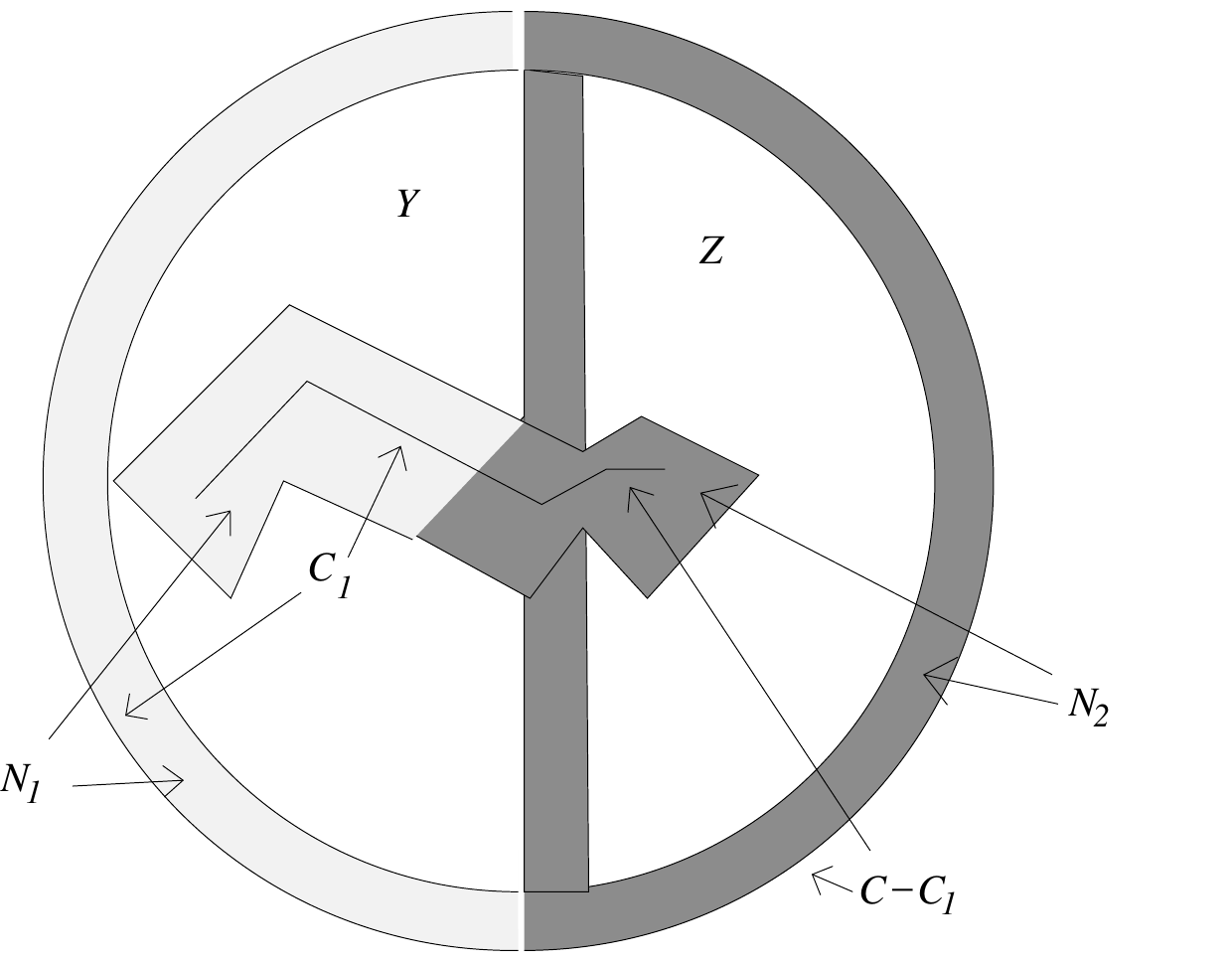}
\vskip 5pt
Basically, we repeat what has been described above, with $K_{2}$ (triangulating $Z$),
${\overset \bullet Z}\bigcup {\overline {C-C_{1}}}$, $g_1$, and 
$Q_{2} \coloneqq({\overline U}'\times {\overline U}'')\bigcap {\mathcal P}$ playing the
respective 
roles of $K_{1}$, $C_1$, $g$ and $Q_{1}$. There is a subcomplex $J_{2}$ of
$K_{2}$,
and a neighborhood $N_{2}$ of ${\overset \bullet Z}\bigcup {\overline {C-C_{1}}}$ in
$\overline Z$ playing the role
previously played by $J_1$ and $N_1$. We get a map $j_{2}=(k_{2}',k_{2}'')$ with
components  $k_{2}':J_{2}\to {\overline U}'$ and $k_{2}'':J_{2}\to {\overline U}''$. 
The construction (imitating what is given in detail in Section \ref{2.2})  means that this
map
$j_{2}$ extends continuously to ${\overline Z}\bigcup C$ agreeing with $j_{1}$ on 
${\overset \bullet Z}\bigcup C$. The extension to
the rest of $Z$ works as before. Thus we get maps  $h$ and $h''$ defined on all of $B$
satisfying
(i)-(v). This shows that $g$ is homotopic, rel $\del B$, to a map whose image misses
${\overline
U}''$.
Summarizing:
\vskip 5pt
\begin{thm}For all $n$, $\pi _{n}(\upset {\overline U}, {\overline U}\bigcup{\overline
U'})$ is
trivial.
\end{thm}
\vskip 5pt
As in Section \ref{2.2} this leads to:
\begin{cor}\label{open3}
The set $\upset {U}$ is weakly contractible.
\end{cor}
\vskip 5pt
\begin{rem} It was important to handle ${\overline U}$ before ${\overline
U}'$ since we needed $C_1$ to be compact. This is a simple instance of
what is a more serious issue in the general case, treated in Section \ref{general}, where
we need
a careful ordering of $\mathcal R$ in order to ensure that at each stage
we are dealing with a compact set.  
\end{rem}


\subsection{Contractibility of pre-images of basis elements}\label{3.1}

As a special case of what was defined in Section \ref{order}, recall that the order
complex is the
set $\Delta ({\mathcal X}) \subseteq X_{1}*X_{2}*X_{3}$
consisting of all points $z=\sum\limits_{i=1}^{3}t_{i}x_{i}$ such that when 
$i,j\in \text{supp }z$ then $x_{i}<x_{j}$ in $\mathcal X$. In particular, every $x_i$ lies
in this
set. The Comparison Map $f \colon \Delta ({\mathcal X})\to {\mathcal X}$ is defined by 
$f(z)=\text{max }\{x_{i}\mid i\in \text{supp }(z)\}$. 
\vskip 5pt

We first consider $U\subseteq X_1$. We describe the set $f^{-1}(\upset {\overline
U})\subseteq
X_{1}*X_{2}*X_{3}$ in some detail. Certainly it  includes $\Delta (\upset {\overline
U})$.
Imitating what was done in Section \ref{2.1}, we let $R^{o}$ be the union of a
certain collection of half-open segments, where the missing point of
each segment is its initial point as measured by the poset ${\mathcal
R}=[3]$. These segments are described as follows. In every case it is to be understood
that $x_1<x_2<x_3$ whenever these
comparisons make sense\footnote{{\it Notation:} $[x,y]$ stands for all the points
$(1-t)x+ty$ on
the segment joining $x$ to $y$.}:

\begin{enumerate}[(i)]
\item $(x_{1}, x_{2}]$ where $x_{1}\notin {\overline U}, x_{2}\in {\overline
U}^{(2)}$, and
$x_{2}$ is not comparable to any member of ${\overline U}^{(3)}$;
\item $(x_{1}, x_{3}]$ where $x_{1}\notin {\overline U}, x_{3}\in {\overline
U}^{(3)}$, and
there is no $x_{2}\in {\overline U}^{(2)}$ with $x_1<x_2<x_3$; 
\item $(x_{1}, [x_{2}, x_{3}]]$ where $x_{1}\notin {\overline U}, x_{2}\in {\overline
U}^{(2)}$, and $x_{3}\in {\overline U}^{(3)}$;  
\item $([x_{1}, x_{2}], x_{3}]$ where $x_{1}\notin {\overline U}$, $x_{2}\notin
{\overline
U}^{(2)}$, and $x_{3}\in {\overline U}^{(3)}$;
\item $(x_{2}, x_{3}]$ where $x_{2}\notin {\overline U}^{(2)}, x_{3}\in {\overline
U}^{(3)}$,
and there is no $x_{1}\in {\overline U}$ with $x_1<x_2$. 
\end{enumerate}

\includegraphics[width=9cm]{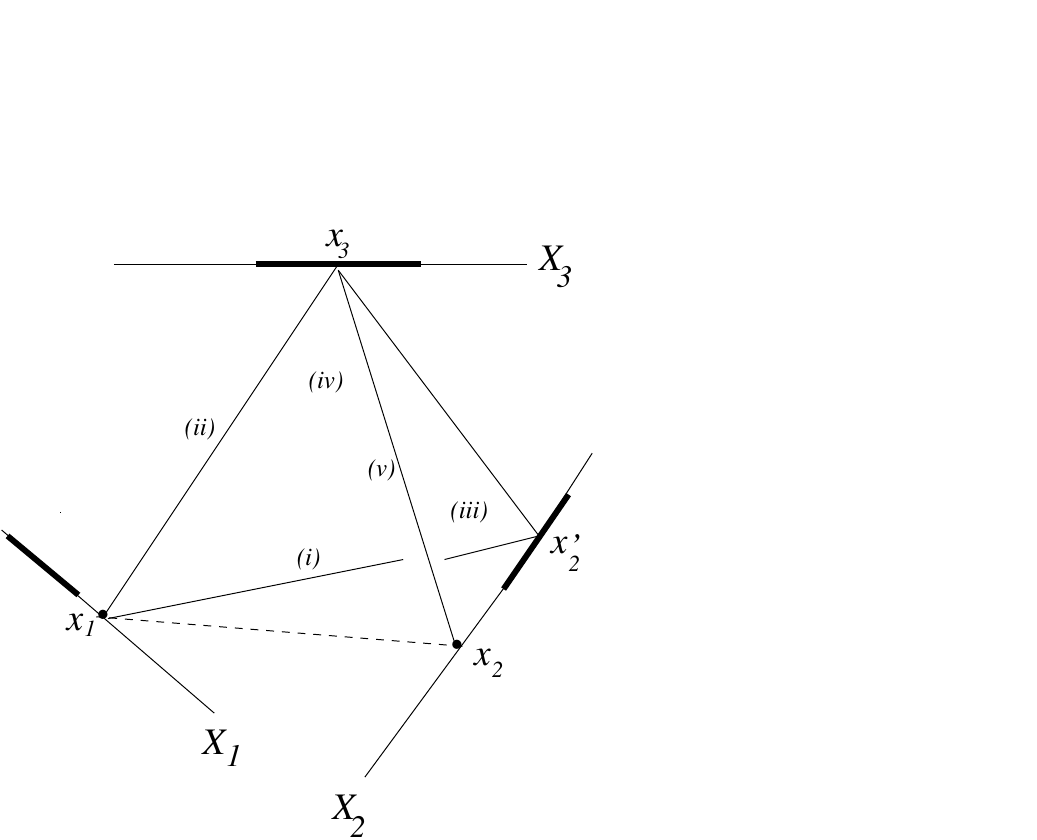}

\vskip 5pt
Then
$$f^{-1}(\upset {\overline U})=\Delta (\upset {\overline U}) \cup R^{o}.$$ 
\vskip 5pt
The set $X_{2}\coprod X_{3}$ is a sub-poset of $\mathcal X$, and we denote
its order complex by $\Delta _{23}$. Recall that $\Delta _{23}$ consists
of copies of $X_2$ and $X_3$ together with a segment\footnote{In all
such expressions, the sum of coefficients is understood to be $1$, and a
segment such as this will require an obvious adjustment when $t_{2}=0$ or
$1$.} $t_{2}x_{2}+t_{3}x_{3}$ joining $x_2$ to $x_3$ whenever $x_{2}<x_3$.
\vskip 5pt

This gives rise to another poset ${\mathcal Y} \coloneqq X_{1}\coprod
\Delta _{23}$ with the partial ordering $x_{1}< t_{2}x_{2}+t_{3}x_{3}$
when $x_1<x_2<x_3$; $\mathcal Y$ is given the Up topology. It is a geometric poset
because
$\mathcal X$ is. We denote
the corresponding Comparison Map by $g \colon \Delta ({\mathcal Y})\to
{\mathcal Y}$. Then, as subsets of $X_{1}\ast X_{2}\ast X_{3}$,
we have

$$f^{-1}(\upset {\overline U})=g^{-1}(\upset {\overline U}).$$

Since $g^{-1}(\upset {\overline U})$ is contractible by Proposition \ref{contractible}, we
conclude that  
$f^{-1}(\upset {\overline U})$ is contractible.
\vskip 5pt
Next we consider the case where $U\subseteq X_2$. We have the Comparison Map 
$f_{23} \colon \Delta _{23}\to X_{2}\coprod X_{3}$. By Proposition \ref{contractible},
$f_{23}^{-1}({\upset {\overline U}})$ is contractible. But $f^{-1}({\upset {\overline
U}})$ is
larger in general, as it includes deleted segments and $2$-simplexes having a vertex in
$f_{23}^{-1}({\upset {\overline U}})$, the deletion being the (non-empty) part lying in
$X_1$.
However, as before\footnote{Compare the discussion of $U\in {\mathcal U}_2$ in
Section \ref{2.1}.}, $f_{23}^{-1}({\upset {\overline U}})$ is a strong deformation
retract of
$f^{-1}({\upset {\overline U}})$, so the latter is contractible\footnote{Compare the
discussion preceding Proposition \ref{cone}.}.
\vskip 5pt
Finally, there is the case where $U\subseteq X_3$. In that case, just
as in the previous paragraph, $\upset {\overline U}$ is a strong
deformation retract of $f^{-1}({\upset {\overline U}})$, and therefore
the latter is contractible. Just as Theorem \ref{weakly closed} implies Theorem
\ref{weakly}
so we conclude:
\vskip 5pt
\begin{prop}\label{contractible3} For basis elements $U$, the space $f^{-1}(\upset
{U})$ is
contractible.     
\end{prop}

\section{The General Case}\label{general}
The general case involves a geometric $M$-poset $\mathcal X$ equipped with a mirror
map $\mu :{\mathcal X}\to {\mathcal R}$, where ${\mathcal R}$ is a finite poset.
\vskip 5pt
\subsection{Weak contractibility of basis elements}
In Section \ref{3weak} we showed that when ${\overline U}^{(1)}$ is
the closure of a basis element of $X_1$ then $\uparrow {\overline U}^{(1)}$
is weakly contractible. The method was to consider a map $g \colon
(B,\del B)\to (\upset {\overline U^{(1)}}, {\overline U^{(1)}}\bigcup
{\overline U^{(2)}})$, define $C=g^{-1}({\overline U}^{(1)}\bigcup
{\overline U}^{(2)})$ and $C_{1}=g^{-1}({\overline U}^{(1)})$. We picked
a closed subset $Y$ of $B-C$ associated with $C_1$, and first adjusted
$g$ on $Y\bigcup C$. Then we made a further adjustment on $Z=B-(C_{1}\bigcup Y)$.
\vskip 5pt
If we were dealing with the case ${\mathcal R}=[4]$ we would do the
same thing with one extra step:  $C_{1}$ would be $g^{-1}({\overline
U}^{(1)})$,  $C_{2}$ would be $g^{-1}({\overline U}^{(1)}\bigcup
{\overline U}^{(2)})$ and $C$ would be $g^{-1}({\overline U}^{(1)}\bigcup
{\overline U}^{(2)}\bigcup {\overline U}^{(3)})$. There would be subspaces 
$Y \coloneqq\{x \in B-C\mid d(x,C)=d(x,C_1)\}$ and 
$Z \coloneqq \text{the closure of }\{x \in B-C \mid x\notin Y \text{and }
d(x,C)=d(x,C_2)\}$ 
associated with $C_1$ and $C_2$. The adjustment of
$g$ would be made first on $Y\bigcup C$, then on $Z\bigcup C$ and finally on the rest
of $B-C$.
Thus $g$ would be homotopic, rel $\del B$, to a map whose image
misses ${\overline U}^{(4)}$.
\vskip 5pt
Enough has been said here to indicate that the same is true for 
${\mathcal R}=[m]$; namely, if $U^{(1)}$ is a basis element of $X_1$ then 
$\uparrow U^{(1)}$ is weakly contractible.
\vskip 5pt
To get the same result for an arbitrary finite poset $\mathcal R$, we need
to order its elements with care. Let ${\mathcal R}_{0}$ denote the set of minimal
elements of ${\mathcal R}$. For $i>0$ let
$${\mathcal R}_{i} \coloneqq\{r\in {\mathcal R}\mid \text{the longest chain connecting
} r \text{ to an element of }{\mathcal R}_{0} \text{ has length }i\}.$$ 
We choose a total ordering of  $\mathcal R$ so that, for all $i$, every element of
${\mathcal
R}_i$ comes before any element of ${\mathcal R}_{i+1}$. We call this the 
{\it useful ordering} to distinguish it from the given partial ordering
on $\mathcal R$. The useful ordering ensures a key point: that for any $r$, 
$\bigcup \{X_{s}\mid s \text{ precedes }r\text{ in the useful ordering}\}$ is closed in the
Up
topology on ${\mathcal X}$; this is because elements indexed by the same $i$ are
incomparable. 
\vskip 5pt
We are to show that, for any basis element $U^{(s)}$ in $X_s$, $\upset
{\overline U}^{(s)}$ is weakly contractible. There is no loss of generality in
assuming $s\in \mathcal R_{0}$. Let $m$ be the
maximal element in the useful ordering such that $s\leq m$. We consider a map 
$g \colon (B,\del B)\to (\upset {\overline U}^{(s)}, \upset {\overline U}^{(s)}-{\overline
U}^{(m)})$. For the argument it is important that 
$g^{-1}(\upset {\overline U}^{(s)}-{\overline U}^{(m)})$ be compact, and this is
assured
by the ``key point" in the previous paragraph.

If the  pieces $U^{(t)}$ with $s\leq t<m$ are taken in ascending order with respect 
to the useful ordering, then a 
generalization of the above
discussion for ${\mathcal R}=[3]$ and ${\mathcal R}=[4]$ adjusts $g$, rel
$\del B$, to a map whose image misses ${\overline U}^{(m)}$. In other words, for all
$n$,  
$\pi _{n}(\upset {\overline U}^{(s)}, \upset {\overline U}^{(s)}-{\overline
U}^{(m)})=0$
The homotopy exact sequence then implies 
$\pi _{n}(\upset {\overline U}^{(s)}-{\overline U}^{(m)})=0$. 
By induction on $|{\mathcal R}|$ this gives
\begin{thm}For all $n$ and $s$, $\pi _{n}(\upset {\overline U}^{(s)})=0$. 
\end{thm} 
\vskip 5pt
Just as in Section \ref{2.2} this leads to\footnote{See ``Proof of this implication"
following the statement of Theorem \ref{weakly closed}.}:  
\begin{cor}\label{opengeneral} For all $s\in {\mathcal R}$ the set $\upset {U^{(s)}}$ is
weakly
contractible.
\end{cor}


\subsection{Contractibility of pre-images of basis elements}\label{4.1}
The Comparison Map $f:\Delta ({\mathcal X})\to {\mathcal X}$ is, as before,
$f(z)=\text{max}\{x_{i}\mid i\in \text{ supp}(z)\}$. Let $r\in
{\mathcal R}$ and let $U\subseteq X_r$ be a basis element. 
\vskip 5pt
\begin{prop}\label{preimagegeneral} $f^{-1}(\uparrow {\overline U})$ is contractible.
\end{prop}
\begin{proof}
We imitate Section \ref{3case}, inducting on the cardinality of $\mathcal R$. When 
$\mathcal R$ is a singleton the proposition is trivially true. When $\mathcal R$ has two
elements
the proposition has been proved. We assume $\mathcal R$ has at least three elements.
\vskip 5pt
There are two cases
\vskip 5pt
First assume $r$ is a minimal element of $\mathcal R$. Then
$$f^{-1}(\uparrow {\overline U})=\Delta (\uparrow {\overline U}) \bigcup
R^{o}$$ where $R^o$ is the union of deleted simplexes, each having some
of its vertices outside $\uparrow {\overline U}$ and the remaining vertices
(at least one) in $\uparrow {\overline U}$, the deletions being the
faces lying outside $\uparrow {\overline U}$.  Now, ${\mathcal Z} \coloneqq{\mathcal
X}-X_r$ is a sub-poset. From it we can construct a new poset ${\mathcal
Y} \coloneqq X_{r}\coprod \Delta ({\mathcal Z})$ where the details imitate the
construction given explicitly in Section \ref{3.1}. The Comparison Map
is $g:\Delta ({\mathcal Y})\to {\mathcal Y}$. Then $$f^{-1}(\uparrow
{\overline U})=g^{-1}(\uparrow {\overline U}).$$ By induction and
Corollary \ref{contractible}, $g^{-1}(\uparrow {\overline U})$ is
contractible, so we conclude that $f^{-1}(\uparrow {\overline U})$
is contractible.

\vskip 5pt
In the other case, $r$ is not minimal. Then there is a proper
sub-$M$-poset $\uparrow X_r$ indexed by a proper subset of $\mathcal
R$. By induction on the cardinality of $\mathcal R$, $g^{-1}(\uparrow {\overline
U})$ is contractible, where $g$ denotes the Comparison Map for  $\uparrow
X_r$. Then $$f^{-1}(\uparrow {\overline U})=g^{-1}(\uparrow {\overline
U})\bigcup R^{o}$$ where $R^o$ is the union of deleted simplexes, each
having some of its vertices outside $\uparrow {\overline U}$ and at least one vertex in
$\uparrow{\overline U}$, the deletions being the faces lying  outside $\uparrow
{\overline U}$.
These deleted simplexes deform to $g^{-1}(\uparrow {\overline U})$. So 
$f^{-1}(\uparrow
{\overline U})$ is contractible.  
\end{proof}
\vskip 5pt
As before, we conclude:
\begin{cor}\label{preimagegeneral} 
For basis elements $U$, $f^{-1}(\uparrow U)$ is contractible 
\end{cor}
\vskip 5pt
Together, Corollaries \ref{opengeneral} and \ref{preimagegeneral} conclude the proof of
Theorem
\ref{main}.

\section{Appendix on the shape theory used in this paper.}\label{appendix}

Shape theory is a variant of homotopy theory which works better than
regular homotopy theory when the spaces are not locally nice. Here we
give a brief review with the sole purpose of  amplifying our use of the
theory in Section \ref{2.1}.

\vskip 5pt
Shape theory considers a compact metrizable space $X$ by dealing instead with an
arbitrary 
inverse sequence of compact polyhedra and maps whose inverse limit is homeomorphic to
$X$.
\vskip 5pt

$$
\xymatrix{
&X_{1}
&X_{2}\ar[l]^{h_{1}}
&X_{3}\ar[l]^{h_{2}}
&\ar[l]^{h_{3}}{\cdots}\\
}
$$

\vskip 5pt
In the following infinite diagram of space and maps, the spaces
are compact polyhedra. We denote the inverse limit of the top line
[resp. bottom line] by $X$ [resp.$Y$]. The outer squares are assumed to
be homotopy commutative. If there exist diagonal maps making the entire
diagram homotopy commutative then $X$ and $Y$ are said to be {\it shape
equivalent}\footnote{This is not the general definition but is enough
for our purposes. For more details on what is discussed in this Appendix, see \cite{Gu}.}
\vskip 5pt
$$
\xymatrix{
&X_{1}\ar[d]
&X_{2}\ar[l]\ar[d]
&X_{3}\ar[l]\ar[d]
&\ar[l]{\cdots}\\
&Y_{1}
&Y_{2}\ar[l]\ar[ul]
&Y_{3}\ar[l]\ar[ul]
&\ar[l]\ar[ul]\cdots\\
}
$$

\vskip 5pt

Now consider the special case in which all the maps $h_i$ are homotopy equivalences.
Then we
have the following homotopy commutative diagram, where $h_{i}^{-1}$ stands for a
homotopy
inverse of $h_i$:
$$
\xymatrix{
&X_{1}\ar[d]_{id}
&X_{1}\ar[l]_{id}\ar[d]^{h_{1}^{-1}}
&X_{1}\ar[l]_{id}\ar[d]^{(h_{1}h_{2})^{-1}}
&\ar[l]_{id}{\cdots}\\
&X_{1}
&X_{2}\ar[l]^{h_{1}}\ar[ul]_{h_{1}}
&X_{3}\ar[l]^{h_{2}}\ar[ul]_{h_{1}h_{2}}&\ar[l]^{h_{3}}{\cdots}\\
}
$$

We conclude that in this case $X$ is shape equivalent to the compact
polyhedron $X_1$. Moreover, in the application in Section \ref{2.1}, $X$
itself is known to be a compact polyhedron. It is a basic theorem of shape
theory that {\it two polyhedra are shape equivalent if and only if they
are homotopy equivalent.}

\vskip 5pt
In the context of Theorem \ref{inverse} this discussion explains why, for $A$ compact, 
$\Delta(\uparrow A)$ is homotopy equivalent to each $X_n$ and hence
to $f^{-1}(\uparrow A)$.

\section{Appendix on the discrete case and McCord's work.}\label{discrete}
\vskip 5pt

A quick review of the present paper shows that the proof of the Comparison 
Theorem \ref{main} becomes very easy in the case when each $X_r$ is
discrete. Confining ourselves to the special case dealt with in Section
\ref{2case}, the relevant version of Proposition \ref{inverse} becomes almost
trivial, since the issue of limit points (which required a shape theoretic
argument) is not present, and the need for Proposition \ref{cone} is
gone since pre-images of basic sets $\upset x$ are cones. The issues
dealt with in Section \ref{2.2} become trivial because Proposition \ref{cover}
implies that each singleton $\{x\}$ is a strong deformation retract of
the set $\upset x$.

\vskip 5pt

McCord's Theorem $3$ in \cite{Mc2} implies our Comparison Theorem \ref{main}
in this discrete case, and his proof is indicated by the contents of the
previous paragraph. It should be added, however, that his definition of
$\Delta ({\mathcal X})$ is superficially different from ours. In his work,
the role of $\Delta ({\mathcal X})$ is played by $|K({\mathcal X})|$,
where $K({\mathcal X})$ is the classical order complex defined by the poset
$\mathcal X$, namely: simplexes are finite chains in the poset. However,
it is clear that $|K({\mathcal X})|$ and $\Delta ({\mathcal X})$
are the same in our setting when $\mathcal X$ is discrete.
\vskip 5pt

McCord uses the ``Down topology" rather than the ``Up topology" but
that is of no consequence as the two are dual to one another
and order complexes are preserved under duality.

\vskip 5pt
It should be added that McCord's Theorem $3$ covers cases outside the setting of the
present
paper. 
\section{Appendix on homotopy equivalence}
\vskip 5pt
In Theorem \ref{main} we showed that for geometric $M$-posets $\mathcal X$ the
Comparison
Map $f \colon \Delta ({\mathcal X})\to {\mathcal X}$ is a weak homotopy equivalence.
The
Whitehead Theorem implies that $f$ is a homotopy equivalence if and only if $\mathcal
X$ has
the homotopy type of a $CW$ complex. In this appendix we show that every member of
an
important class of topological posets (with the Up topology) does not have the homotopy
type of
a $CW$ complex.
\vskip 5pt
Let $\mathcal X$ be a topological poset, not necessarily geometric; it is considered with
the Up
topology. Define
an equivalence relation on $\mathcal X$ by $p\sim q$ if and only if there
is a finite sequence $(x_i)$ in $\mathcal X$ with $x_0=p$, $x_n=q$,
and, for each $i\geq 0$, either $x_i \leq x_{i+1}$ or $x_i \geq x_{i+1}$.
By a proof similar to that of Proposition \ref{cover} each equivalence class is seen to lie
in a path
component of ${\mathcal X}$. We say that
$\mathcal X$ has {\it discrete type} if each equivalence class is an
entire path component\footnote{The topological poset $\mathcal X$ has
an underlying discrete poset which we denote by ${\mathcal X}_{\delta}$,
equipped with the Up topology. Comparing  ${\mathcal X}_{\delta}$ with its
(classical) order complex, one easily sees that each equivalence class
is a path component of ${\mathcal X}_{\delta}$. This explains the term
``discrete type".}
\vskip 5pt
\begin{example} For each $n$ the real Grassmann Poset 
$\mathcal{G}_n(\mathbb{{\R}})$ has discrete type.
\end{example}
\begin{lemma}\label{all} Each equivalence class is a  closed subset of $\mathcal X$.
\end{lemma}
\begin{proof} The relation ${\mathcal P}$ in the definition of topological poset is closed
in
${\mathcal X}\times \mathcal X$, so the condition of not being related is an open
condition. It
follows that
every limit point of an equivalence class lies in that equivalence class.
\end{proof}
\vskip 5pt
For any space $Y$ the space of path components, $\pi _{0}(Y)$, occurs as a quotient
space in a
natural way. When $Y$ is a $CW$ complex this is a discrete space.
\vskip 5pt
\begin{thm}\label{ok} Let $\mathcal X$ have discrete type. Then $\mathcal X$ has the
homotopy type of a $CW$ complex if and only if each path component is contractible,
and the
space of path components $\pi _{0}({\mathcal X})$ is discrete.
\end{thm}
\begin{proof} Let $h: K\to {\mathcal X}$ be a homotopy equivalence, where $K$ is a
$CW$
complex. Suppose there is a homotopy inverse $g$ for $h$. The fact
that  $g$ is continuous shows that if $x<y$ then $g(y)=g(x)$; hence $g$ is
constant on equivalence classes (i.e. on path components). The map $g$ induces
a continuous bijection $\pi _{0}({\mathcal X})\to \pi _{0}(K)$, so $\pi _{0}({\mathcal
X})$ is
discrete. And since $h\circ g$ is homotopic to the identity map, each path component of
$\mathcal X$ must be contractible. Conversely, the hypotheses imply that $\mathcal X$
has the
homtopy type of a discrete space.
\end{proof}
\begin{rem} It follows from Theorem \ref{ok} that equivalence classes
are closed. We include Lemma \ref{all} because this is true whether or
not $\mathcal X$ has discrete type. We wonder whether Theorem \ref{ok}
holds for all topological posets.  
\end{rem} 
\vskip 5pt

\bibliographystyle{amsalpha}
\bibliography{paper2024}{}

\end{document}